\documentclass[dvips,ss,preprint]{imsart}
\usepackage{epsfig,latexsym,xspace}
\RequirePackage{natbib}
\RequirePackage[colorlinks,citecolor=blue,urlcolor=blue]{hyperref}
\RequirePackage{hypernat}

 \doi{10.1214/154957804100000000}
 \pubyear{0000}
 \volume{0}
 \firstpage{0}
 \lastpage{0}
 \arxiv{}

\startlocaldefs
\newcommand{\ie}{{\em i.e.\/}\xspace}
\newcommand{\eg}{{\em e.g.\/}\xspace}
\newcommand{\viz}{{\em viz.\/}\xspace}

\newcommand{\etc}{{\em etc.\/}\xspace}

\newcommand{\graf}{{\cal G}}
\newcommand{\E}{{\mbox{E}}}
\newcommand{\pre}{{\rm pr}}
\newcommand{\mo}{{\rm mo}}
\newcommand{\an}{{\rm an}}
\newcommand{\anc}{{\rm an}}
\newcommand{\ma}{{\rm man}}
\newcommand{\nd}{{\rm nd}}
\newcommand{\parents}[1]{{\rm pa}(#1)}

\newcommand{\eqref}[1]{\mbox{(\ref{eq:#1})}}

\newcommand{\appref}[1]{\mbox{Appendix~\ref{sec:#1}}}
\newcommand{\secref}[1]{\mbox{\S$\,$\ref{sec:#1}}}

\newcommand{\figref}[1]{\mbox{Figure~\ref{fig:#1}}}
\newcommand{\lemref}[1]{\mbox{Lemma~\ref{lem:#1}}}

\newcommand{\defref}[1]{\mbox{Definition~\ref{def:#1}}}
\newcommand{\itref}[1]{\mbox{\ref{it:#1}}}

\newcommand{\thmref}[1]{\mbox{Theorem~\ref{thm:#1}}}

\newcommand{\condref}[1]{\mbox{Condition~\ref{cond:#1}}}

\newcommand{\cip}{\mbox{$\perp\!\!\!\perp$}}

\newtheorem{expl}{Example}[section]
\newtheorem{definer}{Definition}[section]

\newtheorem{theorem}{Theorem}[section]
\newtheorem{thm}[theorem]{Theorem}
\newtheorem{algor}{Algorithm}[section]
\newtheorem{lem}[theorem]{Lemma}
\newtheorem{lemma}[theorem]{Lemma}

\newtheorem{cor}[theorem]{Corollary}
\newtheorem{cond}{Condition}[section]
\newtheorem{rem*}{Remark}[section]

\newcommand{\halm}{\hspace*{\fill} $\Box$\par}
\newenvironment{proof}{\noindent {\bf Proof. }}{\halm\vspace{\baselineskip}}

\newenvironment{ex}{\begin{expl}\rm}{\halm\end{expl}}

\newenvironment{defn}{\begin{definer}\rm}{\end{definer}}

\renewcommand{\cite}{\citep}
\newcommand{\textcite}{\citet}

\newcommand{\co}{o} \newcommand{\ce}{e} 
\newcommand{\as}[1]{\mbox{a.s. [$#1$]}}

\newcommand{\cs}{{\cal S}}

\newcommand{\indo}[2]{\mbox{$#1 \,\cip\, #2$}}
\newcommand{\ind}[3]{\mbox{$#1 \, \cip\, #2 \mid #3$}}
\newcommand{\inda}[4]{\mbox{$#1 \, \cip_{\mbox{\scriptsize $#4$}}\, #2 \mid #3$}}

\renewcommand{\pre}[1]{{\rm pre}(#1)}
\newcommand{\pa}[2]{{\rm pa}_{#1}{(#2)}}

\renewcommand{\dag}{{\cal D}}
\endlocaldefs

\begin{document}

\begin{frontmatter}

  \title{Identifying the consequences of dynamic treatment
    strategies:\\
    A decision-theoretic overview
  }
  \runtitle{Identifying dynamic strategies}

  \begin{aug}
    \author{\fnms{A. Philip}
      \snm{Dawid}\ead[label=e1]{apd@statslab.cam.ac.uk}\ead[label=e2,url]{tinyurl.com/2maycn}}
    \address{Centre for Mathematical Sciences\\University of
      Cambridge\\Wilberforce Road\\Cambridge CB3 0WB\\UK\\
      \printead{e1}\\ \printead{e2}} \and \author{\fnms{Vanessa}
      \snm{Didelez}\corref{}\ead[label=e3]{vanessa.didelez@bristol.ac.uk}
      \ead[label=e4,url]{tinyurl.com/2uuteo8}}
    \address{Department of Mathematics\\University of Bristol\\
      University Walk\\
      Bristol BS8 1TW\\UK\\
      \printead{e3}\\ \printead{e4}}
    
    \runauthor{A. P. Dawid and V. Didelez}
  \end{aug}

\begin{abstract}
  We consider the problem of learning about and comparing the
  consequences of dynamic treatment strategies on the basis of
  observational data.  We formulate this within a probabilistic
  decision-theoretic framework.  Our approach is compared with related
  work by Robins and others: in particular, we show how Robins's
  `$G$-computation' algorithm arises naturally from this
  decision-theoretic perspective.  Careful attention is paid to the
  mathematical and substantive conditions required to justify the use
  of this formula.  These conditions revolve around a property we term
  {\em stability\/}, which relates the probabilistic behaviours of
  observational and interventional regimes.  We show how an assumption
  of `sequential randomization' (or `no unmeasured confounders'), or
  an alternative assumption of `sequential irrelevance', can be used
  to infer stability.  Probabilistic influence diagrams are used to
  simplify manipulations, and their power and limitations are
  discussed.  We compare our approach with alternative formulations
  based on causal DAGs or potential response models.  We aim to show
  that formulating the problem of assessing dynamic treatment
  strategies as a problem of decision analysis brings clarity,
  simplicity and generality.
\end{abstract}

\begin{keyword}[class=AMS]
  \kwd[Primary ]{} \kwd{62C05} \kwd[; secondary ]{62A01}
\end{keyword}

\begin{keyword}
  \kwd{Causal inference} \kwd{$G$-computation} \kwd{Influence diagram}
  \kwd{Observational study} \kwd{Potential response} \kwd{Sequential
    decision theory} \kwd{Stability}
\end{keyword}

\received{\smonth{10} \syear{2010}}

\tableofcontents

\end{frontmatter}

\section{Introduction}
\label{sec:introduction}

Many important practical problems involve sequential decisions, each
chosen in the light of the information available at the time,
including in particular the observed outcomes of earlier decisions.
As an example, consider long-term anticoagulation treatment, as often
given after events such as stroke, pulmonary embolism or deep vein
thrombosis.  The aim is to ensure that the patient's prothrombin time
(INR) is within a target range (which may depend on the diagnosis).
Patients on this treatment are monitored regularly, and when their INR
is outside the target range the dose of anticoagulant is increased or
decreased, so that the dose at any given time is a function of the
previous INR observations.  Despite the availability of limited
guidelines for adjusting the dose, the quality of anticoagulation
control achieved is often poor \cite{henderson}.  Another example is
the question of when to initiate antiretroviral therapy for an
HIV-1-infected patient. The CD4 cell count at which therapy should be
started is a central unresolved issue.  Preliminary findings indicate
that treatment should be initiated when the CD4 cell count drops below
a certain level, \ie\ treatment should be a function of the patient's
previous CD4 count history \cite{sterneetal:09}.

In general, any well-specified way of adjusting the choice of the next
decision (treatment or dose to administer) in the light of previous
information constitutes a dynamic decision (or treatment) {\em
  strategy\/}.  There will typically be an enormous number of
strategies that could be thought of.  Researchers would like to be
able to evaluate and compare these and, ideally, choose a strategy
that is optimal according to a suitable criterion \cite{susan}.  In
many applications, such as the examples given above, it is unlikely
that we will have access to large random samples of patients treated
under each one of the strategies under consideration. At best, the
data available will have been gathered in controlled clinical trials,
but often we will have to content ourselves with data from
uncontrolled observational studies, with, for example, the treatments
being selected by doctors according to informal criteria that we do
not know. The key question we address in the present paper is: Under
what conditions, and how, could the available data be used to
evaluate, compare, and hence choose among, the various decision
strategies?  When a given strategy can be evaluated from available
data it will be termed {\em identifiable\/}.

In principle, our problem can be formulated, represented and solved
using the machinery of sequential decision theory, including decision
trees and influence diagrams \cite{raiffa,oliver:smith:90} --- and
this is indeed the approach that we shall take in this paper.
However, this machinery does not readily provide us with an answer to
the question of when data obtained, for example, from an observational
study will be sufficiently informative to identify a given strategy.
Here, we shall be concerned only with issues around potential biases
in the data, rather than their completeness.  Thus wherever necessary
we suppose that the quantity of data available is sufficient to
estimate, to any desired precision, the parameters of the process that
actually produced those data.  However, that process might still
differ from that in the new decision problem at hand.  We shall
therefore propose simple and empirically meaningful conditions (which
can thus be meaningfully criticised) under which it is appropriate and
possible to make use of the available parameter estimates, and we
shall develop formulae for doing this.  These conditions will be
termed {\em stability\/} due to the way they relate observational and
interventional regimes. We shall further discuss how one might justify
this stability condition by including unobservable variables into the
decision theoretic framework, and by using influence diagrams.

Our proposal is closely related to the seminal work of Robins
\cite{jr:mm,jr:cma,jmr:hsrm,jmr:lns}.  Much of \textcite{jr:mm} takes
an essentially decision theoretic approach, while also using the
framework of structured tree graphs as well as potential responses
(and later using causal direct acyclic graphs (DAGs), see
\textcite{jmr:lns}).  He shows that under conditions linking
hypothetical studies, where the different treatment strategies to be
compared are applied, identifiability can be achieved. Robins calls
these conditions {\em sequential randomization\/} (and later {\em no
  unmeasured confounding\/}, see \eg\ \textcite{robins:bka92}).  While
these are often formalised using potential responses, a closer
inspection of \textcite{jr:mm} (or especially \textcite{jmr:lns})
reveals that all that is needed is an equality of conditional
distributions under different regimes, which is what our stability
conditions state explicitly.  Furthermore, \textcite{jr:mm} introduces
the $G$-computation algorithm as a method to evaluate a sequential
strategy, and contrasts it with traditional regression approaches that
yield biased results even when stability or sequential randomization
holds \cite{robins:bka92}.  We shall demonstrate below that, assuming
stability, this $G$-computation algorithm arises naturally out of our
decision-theoretic analysis, where it can be recognized as a version
of the fundamental `backward induction' recursion algorithm of dynamic
programming.

\subsection{Conditional independence}
\label{sec:language}
The technical underpinning for our decision-theoretic formulation is
the application of the language and calculus of {\em conditional
  independence\/} \cite{apd:CIST,apd:infdiags} to relate observable
variables of two types: `random' variables and `decision' (or
`intervention') variables.  This formalism is used to express
relationships that may be assumed between the probabilistic behaviour
of random variables under differing regimes (\eg, observational and
interventional).  Nevertheless, although it does greatly clarify and
simplify analysis, this particular language is not indispensable:
everything we do could, if so desired, be expressed directly in terms
of relationships between probability distributions for observable
variables.  Thus no essential additional ingredients are being added
to the standard formulation of statistical decision theory.

In many cases the conditional independence relations we work with can
be represented by means of a graphical display: the {\em influence
  diagram\/} (ID).  Once again, although enormously helpful this is,
in a formal sense, only an optional extra.  Moreover, although we pay
special attention to problems that can be represented by influence
diagrams, there are yet others, still falling under our general
approach, where this is not possible.

Inessential though these ingredients are, we nevertheless suggest that
it is well worth the effort of mastering the basic language and
properties, both algebraic and graphical, of conditional independence.
In particular, these allow very simple derivations of the logical
consequences of assumptions made \cite{apd:CIST,ldll}.

\subsection{Overview}
\label{sec:overview}
In \S\S~\ref{sec:multi} and \ref{sec:regimes} we set out the basic
ingredients of our problem and our notation.
Section~\ref{sec:estimation} identifies a simple recursion that can be
used to calculate the consequence of applying a given treatment regime
when the appropriate probabilistic ingredients are available.  In
\S~\ref{sec:ingred} we consider how these ingredients might be come
by, and show that the simple {\em stability\/} condition mentioned
above allows estimation of these ingredients --- and thus, by
application of the procedure of {\em $G$-recursion\/}, of the overall
consequence.  In \S\S~\ref{sec:ext} and \ref{sec:ID} we consider how
one might justify this stability condition, starting from a position
(`extended stability') that might sometimes be more defensible, and
relate various sets of sufficient conditions for this to properties of
influence diagrams.  Section~\ref{sec:moregen} develops more general
conditions, similar to \textcite{jr:cma} and \textcite{jmr:lns}, under
which $G$-recursion can be justified, while \S~\ref{sec:vdproof}
addresses the question of finding an ordering of the involved
variables suitable to carry out $G$-recursion.  Finally
\S\ref{sec:potresp} shows how analyses based on the alternative
formalism of {\em potential responses\/} can be related mathematically
to our own development.

\section{A multistage decision problem}
\label{sec:multi}
We are concerned with a sequential data-gathering and decision-making
process, progressing through a discrete sequence of stages.  The
archetypical context is that of a sequence of medical treatments
applied to a patient over time, each taking into account any interim
responses or adverse reactions to earlier treatments, such as the
anticoagulation treatment for stroke patients or the decision of when
to start antiretroviral therapy for HIV patients.  We shall sometimes
use this language.

Associated with each patient are two sets of variables: ${\cal L}$,
the set of {\em observable variables}, and ${\cal A}$, the set of {\em
  action variables\/}.  The variables in ${\cal A}$ can, in principle,
be manipulated by external intervention, while those in ${\cal L}$ are
generated and revealed by Nature.  The variables in ${\cal L} \cup
{\cal A}$ are termed {\em domain variables\/}.  There is a
distinguished variable $Y\in{\cal L}$, the {\em response variable\/},
of special concern.

A specified sequence ${\cal I}:= (L_1, A_1,\ldots, L_N, A_N, L_{N+1}
\equiv Y)$, where $A_i\in{\cal A}$ and the $L_i$ are disjoint subsets
of ${\cal L}$, defines the {\em information base\/}.  The
interpretation is that the variables arise or are observed in that
order; $L_i$ represents (possibly multivariate, generally
time-dependent) patient characteristics or other variables over which
we have no control, observable between times $i-1$ and $i$; $A_i$
describes the treatment action applied to the patient at time $i$; and
$Y$ is the final `response variable' of primary interest.

For simplicity we suppose throughout that all these variables exist
and can be observed for every patient.  Thus we do not directly
consider cases where, \eg, $Y$ is time to death, which might occur
before some of the $L$'s and $A$'s have had a chance to materialize.
However our analyses could readily be elaborated to handle such
extensions.

When the aim is to control $Y$ through appropriate choices for the
action variables $(A_i)$, any principled approach will involve making
comparisons, formal or informal, between the implied distributions of
$Y$ under a variety of possible strategies for choosing the $(A_i)$.
For example, we might have specified a loss $L(y)$ associated with
each outcome $y$ of $Y$, and desire to minimise its expectation
$\E\{L(Y)\}$.%
\footnote{%
  \label{fn:loss}
  Realistically the loss could also depend on the values of
  intermediate variables, \eg\ if these relate to adverse drug
  reactions.  Such problems can be treated by redefining $Y$ as the
  overall loss suffered (at any rate so long as this loss does not
  depend on other, unobserved, variables.)  } %
Any such decision problem can be solved as soon as we know the
relevant distributions for $Y$ \cite[Section~6]{apd:cinfer}.

The simplest kind of strategy is to apply some fixed pre-defined
sequence of actions, irrespective of any observations on the patient:
we call this a {\em static\/} or {\em unconditional\/} strategy
(\textcite{pearl:book} terms it {\em atomic\/}).  However in realistic
contexts static strategies, which do not take any account of accruing
information, will be of little interest.  In particular, under a
decision-theoretically optimal strategy the action to be taken at any
stage must typically be chosen to respond appropriately to the data
available at that stage \cite{jmr:hsrm,susan}.

A {\em non-randomized dynamic treatment strategy\/} (with respect to a
given information base ${\cal I}$) is a rule that determines, for each
stage $i$ and each configuration (or {\em partial history\/}) $h_i :=
(l_1, a_1, \ldots, a_{i-1}, l_i)$ for the variables $(L_1, A_1,
\ldots, A_{i-1}, L_i)$ available prior to that stage, the value $a_i$
of $A_i$ that is then to be applied.
  
Any decision-theoretically optimal strategy can always be chosen to be
non-randomized.  Nevertheless, for added generality we shall also
consider {\em randomized\footnote{More correctly, these correspond to
    what are termed {\em behavioral rules\/} in decision theory
    \cite{ferguson:book}} dynamic treatment strategies\/}.  Such a
strategy determines, for each stage $i$ and associated partial history
$h_i$, a probability distribution for $A_i$, describing the random way
in which the next action $A_i$ is to be generated.  When every such
randomization distribution is degenerate at a single action this
reduces to a non-randomized strategy.

Suppose now we wish to compare a number of such strategies.  If we
knew or could estimate the full probabilistic structure of all the
variables under each of these, we could simply calculate and compare
directly the various distributions for the response $Y$.  As outlined
in the introduction, our principal concern in this paper is how to
obtain such distributional knowledge, when in many cases the only data
available will have been gathered under purely observational or other
circumstances that might be very different from the strategies we want
to compare.  To clarify the potential difficulties, consider a
statistician or scientist S, who has obtained data on a collection of
variables for a large number of patients.  She wishes to use her data,
if possible, to identify and compare the consequences of various
treatment interventions or policies that might be contemplated for
some new patient.  A major complication, and the motivation for much
work in this area, is that S's observational data will often be
subject to `confounding'.  For example, S's observations may include
actions $(A_i)$ that have been determined by a doctor D, partly on the
basis of additional private information D has about the patient, over
and above the variables S has measured.  Then knowledge of the fact
that D has selected an act $A_i = a_i$, by virtue of that being
correlated with unobserved private information D has that may also be
predictive of the response $Y$, could affect the distribution of $Y$
in this {\em observational\/} regime in a way different from what
would occur if D had no such private information, or if S had herself
chosen the value of $A_i$.  In particular, without giving careful
thought to the matter we cannot simply assume that probabilistic
behaviour seen under the observational regime will be directly
relevant to other, \eg interventional, regimes of interest.

\section{Regimes and consequences}
\label{sec:regimes}
In general, we consider the distribution of all the variables in the
problem under a variety of different regimes, possibly but not
necessarily involving external intervention.  For example, these might
describe different locations, time-periods, or contexts in which
observations can be made.  For simplicity we suppose that the domain
variables are the same for all regimes.  Formally, we introduce a {\em
  regime indicator\/}, $\sigma$, taking values in some set $\cs$,
which specifies which regime is under consideration --- and thus which
(known or unknown) joint distribution over the domain variables ${\cal
  L}\cup{\cal A}$ is operating.  Thus $\sigma$ has the logical status
of a parameter or decision variable, rather than a random variable.
We think of the value $s$ of $\sigma$ as being determined externally,
before any observations are made; all probability statements about the
domain variables must then be explicitly or implicitly conditional on
the value of $\sigma$.  We use \eg $p( y \mid x \,;\, s)$ to denote
the conditional density for $Y$, at $y$, given $X = x$, under regime
$\sigma = s$.  In order to side-step measure-theoretic subtleties, we
shall confine attention to the case that all variables considered are
discrete; in particular, the terms `distribution' or `density' should
be interpreted as denoting a probability mass function.  However, the
basic logic of our arguments does extend to more general cases (albeit
with some non-trivial technical complications to handle null events.)

If we know $p(y;s)$ for all $y$, we can determine, for any function
$k(\cdot)$, the expectation $\E\{k(Y);s\}$.  Often we shall be
interested in one or a small number of such functions, \eg a loss
function $k(y) \equiv L(y)$.  For definiteness we henceforth consider
a fixed given function $k(Y)$, and use the term {\em consequence\/} of
$s$ to denote the expectation $\E\{k(Y);s\}$ of $k(Y)$ when regime $s$
is followed.

More generally we might wish to focus attention on a subgroup
(typically defined in terms of the pre-treatment information $L_1$),
and compare the various `conditional consequences', given membership
of the subgroup.  Although we do not address this directly here, it is
straightforward to extend our unconditional analysis to this case.

\subsection{Inference across regimes}
\label{sec:across}

In the most usual and useful situation, $\cs=\{\co\}\cup \cs^*$, where
$\co$ is a particular {\em observational regime\/} under which data
have been gathered, and $\cs^*$ is a collection of contemplated {\em
  interventional strategies\/} with respect to the information base
$(L_1, A_1,\ldots, L_N, A_N, Y)$.  We wish to use data collected under
the observational regime $\co$ to identify the consequence of
following any of the strategies $\ce\in\cs^*$.  This means we need to
make inference strictly beyond the available data to what would
happen, in future cases, under regimes that we have not been able to
observe in the past.

It should be obvious, but nonetheless deserves emphasis, that we can
not begin to address this problem without assuming some relationships
between the probabilistic behaviour of the variables across the
differing regimes, both observed and unobserved.  Inferences across
regimes will typically be highly sensitive to the assumptions made,
and the validity of our conclusions will depend on their
reasonableness.  Although in principle any such assumptions are open
to empirical test, using data gathered under all the regimes involved,
this will often be impossible in practice.  In this case, while it is
easy to make assumptions, it can be much harder to justify them.  Any
justification must involve context-dependent considerations, which we
can not begin to address here.  Instead we simply aim to understand
the logical consequences of making certain assumptions.  One message
that could be drawn is: if you don't like the consequences, rethink
your assumptions.

\section{Evaluation of consequences}
\label{sec:estimation}
Writing \eg\ $(L_1, L_2)$ for $L_1 \cup L_2$, we denote $(L_1, \ldots,
L_i)$ by $\overline{L}_i$, with similar conventions for other
variables in the problem.

For any fixed regime $s$, we can specify the joint distribution of
$(L_1,A_1, \ldots, L_N, A_N, Y)$, when ${\sigma} = {s}$, in terms of
its sequential conditional distributions for each variable, given all
earlier variables.  These comprise:
\begin{enumerate}
\item
  \label{it:li}
  $p(l_i \mid \overline l_{i-1}, \overline a_{i-1}\,;\, s)$ for $i =
  1, \ldots, N$.
\item
  \label{it:ai} $p(a_i \mid \overline l_{i}, \overline a_{i-1}\,;\,
  s)$ for $i = 1, \ldots,N$.
\item
  \label{it:y}
  $p(y \mid \overline l_N, \overline a_N\,;\, s)$.
\end{enumerate}
Note that \itref{y} can also be considered as the special case of
\itref{li} for $i = N+1$.

With $l_{N+1} \equiv y$, we can factorize the overall joint density
as:
\begin{eqnarray}
  \label{eq:factor0}
  p(y, \overline{l}, \overline a\,;\, {s}) = \left\{\prod_{i=1}^{N+1}
    p(l_i \mid \overline{l}_{i-1},  \overline{a}_{i-1}\,;\, {s})\right\} \times
  \left\{\prod_{i=1}^N p(a_i \mid \overline{l}_{i},
    \overline{a}_{i-1}\,;\, {s})\right\}.
\end{eqnarray}

If we know all the terms in \eqref{factor0}, we can simply sum out
over all variables but $l_{N+1} \equiv y$ to obtain the desired
distribution $p(y;s)$ of $Y$ under regime $s$, from which we can in
turn compute the consequence $\E\{k(Y);s\}$.

Alternatively, and more efficiently, this calculation can be
implemented recursively, as follows.  Let $h$ denote a partial
history, of the form $(\overline l_i, \overline a_{i-1})$ or
$(\overline l_i, \overline a_i)$ ($0 \leq i \leq N)$.  We also include
the `null' history $\emptyset$, and `full' histories $(\overline l_N,
\overline a_N, y)$.  We denote the set of all partial histories by
${\cal H}$.  Fixing the regime $s$, define a function $f$ on ${\cal
  H}$ by:
\begin{equation}
  \label{eq:fi0}
  f(h) := \E\{k(Y) \mid h\,;\, s\}.
\end{equation}
Simple application of the laws of probability yields:
\begin{eqnarray}
  \label{eq:mainrecura}
  f(\overline l_i, \overline a_{i-1}) &=& \sum_{a_i} p(a_i  \mid
  \overline l_{i}, \overline a_{i-1}\,;\, s)  
  \times f(\overline l_i, \overline a_i)\\
  \label{eq:mainrecurl}
  f(\overline l_{i-1}, \overline a_{i-1})& =& \sum_{l_i}
  p(l_i  \mid  \overline l_{i-1}, \overline a_{i-1}\,;\,  s) \times
  f(\overline l_i, \overline a_{i-1}).  
\end{eqnarray}
For $h$ a full history $(\overline l_N, \overline a_N, y)$, we have
$f(h) = k(y)$.  Using these as starting values, by successively
implementing \eqref{mainrecura} and \eqref{mainrecurl} in turn,
starting with \eqref{mainrecurl} for $i = N+1$ and ending with
\eqref{mainrecurl} for $i=1$, we step down through ever shorter
histories until we have computed $f(\emptyset) = \E\{k(Y)\,;\,s\}$,
the consequence of regime $s$.\footnote{More generally (see
  footnote~\ref{fn:loss}), we could consider a function $Y^*$ of
  $(\overline L_N,\overline A_N, Y)$.  Starting now with $f(\overline
  l_N, \overline a_N, y) := Y^*(\overline l_N, \overline a_N, y)$, we
  can apply the identical steps to arrive at $f(\emptyset) =
  \E\{Y^*\,;\,s\}$.  In particular we can evaluate the expected
  overall loss under $s$, even when the loss function depends on the
  full sequence of variables.}

The recursion expressed by \eqref{mainrecura} and \eqref{mainrecurl}
is exactly that underlying the `extensive form' analysis of sequential
decision theory (see \eg\ \textcite{raiffa}).  In particular, under
suitable further conditions we can combine this recursive method for
evaluation of consequences with the selection of an optimal strategy,
when it becomes {\em dynamic programming\/}.  This `step-down
histories' approach also applies just as readily to more general
probability or decision trees, where the length of the history, and
even the variables entering into it, can vary with the path followed.
We do not consider such extensions here, but they raise no new issues
of principle.

When $s$ is a non-randomized strategy, the distribution of $A_i$ given
$\overline L_i = \overline l_i$, when $\sigma= s$, is degenerate, at
$a_i = g_i = g_i(\overline l_i\,;\, s)$, say, and the only randomness
left is for the variables $(L_1, \ldots, L_N, Y)$.  We can now
consider $f(h)$ as a function of only the $(l_i)$ appearing in $h$,
since, under $s$, these then determine the $(a_i)$.  Then
\eqref{mainrecura} holds automatically, while \eqref{mainrecurl}
becomes:
\begin{equation}
  \label{eq:determrecur}
  f(\overline l_{i-1}) = \sum_{l_i}\,
  p(l_i  \mid  \overline l_{i-1}, \overline g_{i-1};
  s)  \times f( \overline l_i).
\end{equation}
When, further, the regime $s$ is static, each $g_i$ in the above
expressions reduces to the fixed action $a_i^*$ specified by $s$.

We remark that the conditional distributions in \itref{li}--\itref{y}
and \eqref{fi0} are undefined when the conditioning event has
probability 0 under $s$.  The overall results of recursive application
of \eqref{mainrecura} and \eqref{mainrecurl} will not depend on how
such ambiguities are resolved.  However, for later convenience we
henceforth assume that $f(h)$ in \eqref{fi0} is defined as 0 whenever
$p(h\,;\, s) = 0$.  Note that this property is preserved under
\eqref{mainrecura} and \eqref{mainrecurl}.

\section{Identifying the ingredients}
\label{sec:ingred}

In order for the statistician S to be able to apply the above
recursive method to calculate the consequence of some contemplated
regime $s$, she needs to know all the ingredients \itref{li},
\itref{ai} and \itref{y}.  How might such knowledge be attained?

\subsection{Control strategies}
\label{sec:control}
Consider first the term $p(a_i \mid \overline l_{i}, \overline
a_{i-1}\,;\, s)$ in \itref{ai}, as needed for \eqref{mainrecura}.  It
will often be the case that for the regimes $s$ of interest this is
known {\em a priori\/} to the statistician S for all $i$.  For
instance we might be interested in strategies for initiating
antiretroviral treatment of HIV patients as soon as the CD4 count has
dropped below a given value $c$. The strategy therefore fully
determines the value of the binary $A_i$ given the previous covariate
history $\overline l_{i}$ as long as this includes information on the
CD4 counts.  In such a case we shall call $s$ a {\em control
  strategy\/} (with respect to the information base ${\cal I} = (L_1,
A_1,\ldots, L_N, A_N, Y)$).  In particular this will typically be the
case when $s$ is a (possibly randomized) dynamic strategy, as
introduced in \secref{multi}.

\subsection{Stability}
\label{sec:simple}
More problematic is the source of knowledge of the conditional density
$p(l_i \mid \overline l_{i-1}, \overline a_{i-1}\,;\, s)$ in
\itref{li} as required for \eqref{mainrecurl} (including, as a special
case, that of $p(y \mid \overline l_N, \overline a_N\,;\, s)$ in
\itref{y}).

If we observed many instances of regime $s$, we may be able to
estimate this directly; but typically we will be interested in
assessing the consequences of various contemplated regimes (\eg\ 
control strategies) that we have never yet observed.  The problem then
becomes: under what conditions can we use probability distributions
assessed under one regime to deduce the required conditional
probabilities, \itref{li} and \itref{y}, under another?

In the application of most interest, we have $\cs=\{\co\}\cup \cs^*$,
where $\co$ is an observational regime under which data have been
gathered, and $\cs^*$ is a collection of contemplated interventional
strategies.  If we can use data collected under the observational
regime $\co$ to identify the consequence of following any of the
strategies $\ce\in\cs^*$, we will be in a position to compare the
consequences of different interventional strategies (and thus, if
desired, choose an optimal one) on the basis of data collected in the
single regime $\co$.

In general, the distribution of $L_i$ given $(\overline L_{i-1},
\overline A_{i-1})$ will depend on which regime is in operation.  Even
application of a control strategy might well have effects on the joint
distribution of all the variables, beyond the behaviour it directly
specifies for the actions.  For example, consider an educational
experiment in which we can select certain pupils to undergo additional
home tutoring.  Such an intervention can not be imposed without
subjecting the pupil and his family to additional procedures and
expectations, which would probably be different if the decision to
undergo extra tutoring had come directly from the pupil, and possibly
different again if it had come from the parents.  Consequently we can
not necessarily assume that the distribution of $L_i$ given
$(\overline L_{i-1}, \overline A_{i-1})$ assessed under the
observational regime will be the same as that for an interventional
strategy, or that it would be the same for different interventional
strategies.

It will clearly be helpful when we {\em can\/} impose this assumption
--- and so be able to identify the required interventional
distributions of $L_i$ given $(\overline L_{i-1}, \overline A_{i-1})$
with those assessed under the observational regime.  We formalize this
assumption as follows:

\begin{defn}
  \label{def:expgen}
  We say that the problem exhibits {\em simple stability\/}, with
  respect to the information base ${\cal I} = (L_1,A_1, \ldots, L_N,
  A_N, Y)$ and the set $\cs
  $ of regimes if, with $\sigma$ denoting the non-random regime
  indicator taking values in $\cs$:
  \begin{equation}
    \label{eq:expci2}
    \ind {L_i} {\sigma} {(\overline L_{i-1}, \overline A_{i-1})} \quad (i =
    1, \ldots, N+1).
  \end{equation}
\end{defn}

Here and throughout, we use the notation and theory of {\em
  conditional independence\/} introduced by \textcite{apd:CIST}, as
generalized as in \textcite{apd:infdiags} to apply also to problems
involving decision or parameter variables.  In words,
condition~\eqref{expci2} asserts that the stochastic way in which
$L_i$ arises, given the previous values of the $L$'s and $A$'s, should
be the same, irrespective of which regime in $\cs$ is in operation.
More precisely, expressed in terms of densities, \eqref{expci2}
requires that, for each $i = 1, \ldots, N+1$, there exist some common
conditional density specification $q(L_i = l_i \mid \overline{L}_{i-1}
= \overline{l}_{i-1}, \overline{A}_{i-1} = \overline{a}_{i-1})$ such
that, for each $s\in\cs$,
\begin{equation}
  \label{eq:as}
  p(L_i = l_i \mid \overline{L}_{i-1} = \overline{l}_{i-1},
  \overline{A}_{i-1} =  \overline{a}_{i-1}; s) = 
  q(L_i = l_i \mid \overline{L}_{i-1} = \overline{l}_{i-1},
  \overline{A}_{i-1} =  \overline{a}_{i-1})
\end{equation}
{\em whenever the conditioning event has positive probability under
  regime $s$.\/}

As will be described further in \secref{ID} below, it is often helpful
(though never essential) to represent conditional independence
properties graphically, using the formalism of {\em influence
  diagrams\/} (IDs): such diagrams have very specific semantics, and
can facilitate logical arguments by displaying implied properties in a
particularly transparent form \cite{apd:infdiags}.  The appropriate
graphical encoding of property \eqref{expci2} for $i =$ 1, 2 and 3 is
shown in \figref{simpdiag}.  The specific property \eqref{expci2} is
represented by the {\em absence\/} of arrows from $\sigma$ to $L_1$,
$L_2$, and $Y \equiv L_3$.  For general $N$ we simply supplement the
complete directed graph on $(L_1, A_1, \ldots, L_N, A_N, Y)$ with an
additional regime node $\sigma$, and an arrow from $\sigma$ to each
$A_i$.

\begin{figure}[h]
  \begin{center}
    \resizebox{2.8in}{!}{\includegraphics{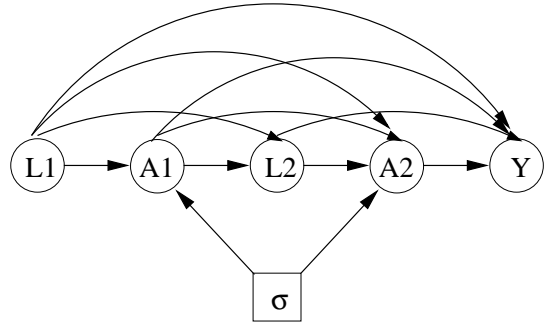}}
    \caption{Influence diagram: stability}
    \label{fig:simpdiag}
  \end{center}
\end{figure}

\subsubsection{Some comments}
\label{sec:comm}

An important question is how we should assess whether
property~\eqref{expci2} holds in any given situation.  It could in
principle be tested empirically, if we could collect data under all
regimes.  In practice this is usually impossible, and other arguments
for or against its appropriateness would be brought to bear.  Whether
or not the simple stability property can be regarded as appropriate in
any application will depend on the overall context of the problem.  In
particular, it will depend on the specific information base involved.
For example, if $\ce$ is a control strategy with respect to S's
information base, and $\co$ an observational regime under which the
doctor D chooses the $(A_i)$ on the basis of private information not
represented in S's information base, possibly associated with $L_i$,
then, for $\cs = \{\co,\ce\}$, we might well expect \eqref{expci2} to
be violated.  This is often described as (potential) {\em
  confounding\/}.

The simple stability property \eqref{expci2} is our version of a
condition termed `sequential randomization' \cite{jr:mm,jmr:lns} or
`no unmeasured confounding'
\cite{robins:bka92,robins/hernan/brumback:marg_struct} or `sequential
ignorability' \cite{jmr:sbs00}.  The connexions become particularly
clear when comparing \eqref{expci2} with the equalities derived in
Theorem~3.1 of \textcite{jmr:lns}, which we consider in more detail in
\secref{conn} below.  These alternative names suggest particular
situations where stability should be satisfied, such as when the data
have been gathered under an observational regime where the actions
were indeed physically sequentially randomized; or when S's
information base contains all the information the doctor D has used in
choosing the $(A_i)$.  However, we emphasise that our
property~\eqref{expci2} can be meaningfully considered even without
referring to any `potential confounder' variables; and that if (as in
\secref{ext} below) we do choose to introduce such further variables
to help us assess whether \eqref{expci2} holds, nevertheless the
property itself must hold or fail quite independently of which
additional variables (if any) are considered.

In any case, because stability is a property of the relationship
between different regimes, it can never be empirically established on
the basis of data collected under only one (\eg, observational)
regime, nor can it be deduced from properties assumed to hold for just
one such regime.

\subsubsection{Positivity}
\label{sec:positivity}
The purpose of invoking simple stability (with respect to $\cs =
\{\co\} \cup \cs^*$) is to get a handle on \eqref{mainrecurl} for an
unobserved interventional strategy $s=\ce\in\cs^*$, using data
obtained in the observational regime $\co$.  Intuitively, under simple
stability we can replace $p(l_i \mid \overline l_{i-1}, \overline
a_{i-1}; \ce)$ by $ p(l_i \mid \overline l_{i-1}, \overline a_{i-1};
\co)$, which is estimable from the observational data.  However, some
care is needed on account of the positivity qualification following
\eqref{as}.  If, for example, we want to assess the consequence of a
static interventional strategy $\ce$, which always applies some
pre-specified action sequence $\overline a^*$, we clearly will be
unable to do so using data from an observational regime in which the
probability of obtaining that particular sequence of actions is zero.
(Pragmatically it may still be difficult to do so if that probability
is non-zero but so small that we are unable to estimate it well from
available observational data.  However we ignore that difficulty here,
supposing that the data are sufficiently extensive that we can indeed
get good estimates of all probabilities under $\co$).

In order to avoid this problem, we impose the {\em positivity
  (absolute continuity) condition\/}:
\begin{defn}
  \label{def:positive}
  We say the problem exhibits {\em positivity\/} if, for any
  $\ce\in\cs^*$, the joint distribution of $(\overline L_N, \overline
  A_N, Y)$ under $P_\ce$ is absolutely continuous with respect to that
  under $P_\co$, \ie\ 
  \begin{equation}
    \label{eq:pos}
    p(E ; \ce) > 0 \Rightarrow p(E ; \co)> 0
  \end{equation}
  for any event $E$ defined in terms of $(\overline L_N, \overline
  A_N, Y)$.  We write this as $P_\ce \ll P_\co$.
\end{defn}
In our discrete set-up, it is clearly enough to demand \eqref{pos}
whenever $E$ comprises a single sequence $(\overline l_N, \overline
a_N, y)$.  Denoting by ${\cal O}$, ${\cal E}$ the sets of partial
histories having positive probability under, respectively, regimes
$\co$ and $\ce$, we can restate \eqref{pos} as
\begin{equation}
  \label{eq:eo}
  {\cal E}\subseteq{\cal O}.
\end{equation}

\subsection{$G$-recursion}
\label{sec:grec0}
Let $\ce\in\cs^*$.  Given enough data collected under $\co$ we can
identify $p(l_i \mid \overline{l}_{i-1}, \overline{a}_{i-1}\,;\, \co)$
($i = 1, \ldots, N+1$) for $(\overline l_{i-1}, \overline
a_{i-1})\in{\cal O}$.  Under simple stability \eqref{as} and
positivity \eqref{eo}, this will also give us $p(l_i \mid
\overline{l}_{i-1}, \overline{a}_{i-1}\,;\, \ce)$ ($i = 1, \ldots,
N+1$) for all $(\overline l_{i-1}, \overline a_{i-1})\in{\cal E}$.
If, further, $\ce$ is a control strategy, then using the known form
for $p(a_i \mid \overline l_{i}, \overline a_{i-1}\,;\, \ce)$
($(\overline l_{i}, \overline a_{i})\in{\cal E}$), we have all the
ingredients to apply \eqref{mainrecura} and \eqref{mainrecurl} and
thus identify the consequence of regime $\ce$ from data collected
under $\co$.

Specifically, we have
\begin{eqnarray}
  \label{eq:standrecura}
  f(\overline l_i, \overline a_{i-1}) &=& \sum_{a_i} p(a_i  \mid
  \overline l_{i}, \overline a_{i-1}\,;\, \ce)  
  \times f(\overline l_i, \overline a_i)\\
  \label{eq:standrecurl}
  f(\overline l_{i-1}, \overline a_{i-1})& =& \sum_{l_i}
  p(l_i  \mid  \overline l_{i-1}, \overline a_{i-1}\,;\,  \co) \times
  f(\overline l_i, \overline a_{i-1}).  
\end{eqnarray}
We start the recursion with
\begin{displaymath}
  f( \overline l_N, \overline a_N) \equiv \E\{k(Y) \mid \overline l_N,
  \overline a_N\,;\, \ce\} = 
  \left\{  
    \begin{array}[c]{ll}
      \E\{k(Y) \mid \overline l_N, \overline a_N\,;\, \co\} & \mbox{if }
      (\overline l_N, \overline a_N)\in{\cal E}\\
      0 & \mbox{otherwise}
    \end{array}
  \right.
\end{displaymath}
(using simple stability for $i= N+1$), and exit with the desired
interventional consequence $f(\emptyset) \equiv \E\{k(Y) \,;\, \ce\}$.

We refer to the above method as {\em $G$-recursion\/}.\footnote{%
  Cases in which simple stability may not hold but we can nevertheless
  still apply $G$-recursion are considered in
  Section~\ref{sec:moregen}.%
}

For the case that $\ce$ is a non-randomized strategy, $G$-recursion
can be based on \eqref{determrecur}, becoming
\begin{equation}
  \label{eq:determstandrecur}
  f(\overline l_{i-1}) = \sum_{l_i}\,
  p(l_i \mid \overline l_{i-1}, \overline g_{i-1}; \co) \times f( \overline l_i),
\end{equation}
starting with $f(\overline l_N) = \E\{k(Y) \mid \overline l_N,
\overline g_N\,;\, \co\}$.  The {\em $G$-computation formula\/}
\cite{jr:mm} is the algebraic formula for $f(\emptyset)$ in terms of
$f(\overline l_N)$ that results when we write out explicitly the
successive substitutions required to perform this recursion.

Finally we remark that, when the simple stability property
\eqref{expci2} holds for $(L_1, A_1,\ldots, L_N, A_N, Y)$, it also
holds for $(L_1, A_1, \ldots, L_N, A_N, Y^*)$, where $Y^*$ is any
function of $(L_1, A_1, \ldots, L_N, A_N, Y)$.  For $i \leq N$ there
is nothing new to show, while \eqref{expci2} for $i = N+1$ follows
easily for $Y^*$ when it holds for $Y$, using general properties of
conditional independence \cite{apd:CIST}.  It is also easy to see that
when positivity, \defref{positive}, holds for $(\overline L_N,
\overline A_N, Y)$ it likewise holds for $(\overline L_N, \overline
A_N, Y^*)$.  Consequently, under the same conditions that allow
$G$-recursion to compute the interventional distribution of $Y$, we
can use it to compute that of $Y^*$.  In particular (see
footnote~\ref{fn:loss}), this will allow us to evaluate the expected
loss of applying $\ce$, even when the loss function depends on all of
$(\overline L_N, \overline A_N, Y)$.

\section{Extended stability}
\label{sec:ext}

We have already alluded to the possibility that, in many applications,
the simple stability assumption \eqref{expci2} might not be easy to
justify directly.  This might be the case, in particular, when we are
concerned about the possibility of `confounding effects' due to
unobserved influential variables.

In such a case we might proceed by constructing a more detailed model,
incorporating a collection ${\cal U}$ of additional, possibly
unobserved, variables; and investigate its implications.  These
unobserved variables might be termed `sequential (potential)
confounders'.  Under certain additional assumptions to be discussed
below, we might then be able to deduce that simple stability does,
after all, apply.  This programme can be helpful when the assumptions
involving the additional variables are easier to justify than
assumptions referring only to the variables of direct interest.  We
here initially express these additional assumptions purely
algebraically, in terms of conditional independence; in \secref{ID} we
shall conduct a parallel analysis utilising influence diagrams to
facilitate the expression and manipulation of the relevant conditional
independencies.

Reasoning superficially similar to ours has been conducted by
\textcite{pearl/robins:uai95} and \textcite{jmr:lns}.  However, that
is mostly based on the assumed existence of a `causal DAG'
representation of the problem.  We once again emphasise that the
simple stability property \eqref{expci2} is always meaningful of
itself, and its truth or falsity can not rely on the possibility of
carrying out such a programme of reduction from a more complex model
including unobservable variables.

\subsection{Preliminaries}

We shall specifically investigate models having a property we term
{\em extended stability\/}.  Such a model again involves a collection
${\cal L}$ of observable domain variables (including a response
variable $Y$) and a collection ${\cal A}$ of action domain variables,
together with a regime indicator variable $\sigma$ taking values in
$\cs = \{\co\} \cup \cs^*$.  But now we also have the collection
${\cal U}$ of unobservable domain variables (for simplicity we suppose
throughout that which variables are observed or unobserved is the same
under all regimes considered).  Let ${\cal I}'$ denote an ordering of
all these observable and unobservable domain variables (typically,
though not necessarily, their time-ordering). As before we assume that
$A_{i-1}$ comes before $A_i$ in this ordering.  We term ${\cal I}'$ an
{\em extended information base\/}. Let $L_i \subseteq{\cal L}$ [resp.,
$U_i \subseteq {\cal U}$] denote the set of observed [resp.,
unobserved] variables between $A_{i-1}$ and $A_i$.

\begin{defn}
  \label{def:es}
  We say that the problem exhibits {\em extended stability\/} with
  respect to the extended information base ${\cal I}'$ and the set
  $\cs$ of regimes if, for $i=1,\ldots, N+1$,
  \begin{equation}
    \label{eq:es}
    \ind {(U_i,L_i)} \sigma ({\overline{U}_{i-1}, \overline{L}_{i-1},
      \overline{A}_{i-1}}). 
  \end{equation}
\end{defn}
(If the ($U_i$) were observable, this would be identical with the
definition of simple stability.)

Under extended stability the marginal distribution of $U_1$ is
supposed the same in both regimes, as is the conditional distribution
of $U_2$ given $(U_1, L_1, A_1)$, \etc\@ Similarly, the distributions
of $L_1$ given $U_1$, of $L_2$ given $(U_1, L_1, A_1, U_2)$,\ldots,
and finally of $Y$ ($=L_{N+1}$) given $(U_1, L_1, A_1,\ldots, U_N,
L_N, A_N)$, are all supposed to be independent of the regime
operating.

There is a corresponding extension of \defref{positive}:
\begin{defn}
  \label{def:extpositive}
  We say the problem exhibits {\em extended positivity\/} if, for any
  $\ce\in\cs^*$, $P_\ce \ll P_\co$ as distributions over $(\overline
  L_N, \overline U_N, \overline A_N, Y)$; that is, $p(E ; \ce) > 0
  \Rightarrow p(E ; \co)> 0$ and any event $E$ defined in terms of
  $(\overline L_N, \overline U_N, \overline A_N, Y)$.
\end{defn}

In many problems, though by no means universally, an extended
stability assumption might be regarded as more reasonable and
defensible than simple stability --- so long as appropriate unobserved
variables ${\cal U}$ are taken into account.  For example, this might
be the case if we believed that, in the observational regime, the
actions were chosen by a decision-maker who had been able to observe,
in sequence, some or all of the variables in the problem, including
possibly the $U$'s; and was then operating a control strategy with
respect to this extended information base, so that, when choosing each
action, he was taking account of all previous variables in this
extended sequence, but nothing else.  But even then, as discussed in
\secref{simple}, the extended stability property is a strong
additional assumption, that needs to be justified in any particular
problem.  And again, because it involves the relationships between
distributions under different regimes, it can not be justified on the
basis of considerations or findings that apply only to one regime.

Unobservable variables can assist in modelling the observational
regime and its relationship with the interventional control regimes
under consideration.  But, because they are unobserved, they can not
form part of the information taken into account by such control
regimes.  Thus we shall still be concerned with evaluating --- using
$G$-recursion when possible --- a regime $\ce$ that is a control
strategy with respect to the {\em observable} information base ${\cal
  I} = (L_1, A_1,\ldots, L_N, A_N, Y)$ as introduced in
\secref{control}.  More specifically, in this more general context we
define:

\begin{cond}[Control strategy]
  The regime $\ce$ is a {\em control strategy\/} if, for $i = 1,
  \ldots, N$,
  \label{cond:cont}
  \begin{equation}
    \label{eq:controlpar}
    \ind {A_i} {\overline{U}_i} {(\overline{L}_i,
      \overline{A}_{i-1}\,;\, \ce)} 
  \end{equation}
  and, in addition, the conditional distribution of $A_i$, given
  $(\overline L_{i}, \overline A_{i-1})$, under regime $\ce$, is known
  to the analyst.
\end{cond}
\condref{cont} expresses the property that, under regime $\ce$, the
randomization distribution or other sources of uncertainty about
$A_i$, given all earlier variables, does not in fact depend on the
earlier unobserved variables; and that this conditional distribution
is known.  The condition will hold, in particular, in the important
common case that, under $\ce$, $A_i$ is fully specified as a function
of previous observables.

\subsection{Stability regained}
\label{sec:regained}

When there are unobservables in the problem, the extended positivity
property of \defref{extpositive} will clearly imply the simple
positivity property of \defref{positive}.  However, even when extended
stability holds, the simple stability property, with respect to the
observable information base $(L_1, A_1, \ldots, L_N, A_N, Y)$ from
which (as is a pragmatic necessity) we have had to exclude the
unobserved variables, will typically fail.  But we can sometimes
incorporate additional background knowledge, most usefully expressed
in terms of conditional independence, to show that it does, after all,
hold.

We now describe two sets of additional sufficient (though not
necessary) conditions, either of which will, when appropriate, allow
us to deduce the simple stability property \eqref{expci2} --- and with
it, the possibility of applying $G$-recursion (ignoring the
unobservable variables), as set out in \secref{grec0}.  The results in
this section can be regarded as extending the analysis of
\textcite{apd:infdiags} \S~8.3 (see also
\textcite{hg/apd:aistats2010}) to the sequential setting.

\subsubsection{Sequential randomization}

It has frequently been proposed (\eg, \textcite{jr:mm,jmr:lns}) that
when, under an observational regime, the actions $(A_i)$ have been
physically (sequentially) randomized, then simple stability
\eqref{expci2} will hold.  Indeed, our concept of simple stability has
also been termed `sequential randomization' \cite{jr:mm}. However we
shall be more specific and restrict the term {\em sequential
  randomization\/} to the special case that we have extended stability
and, in addition, \condref{pa1} below holds.  We shall show that these
properties are indeed sufficient to imply simple stability --- but
they are by no means necessary.

So consider now the following condition:
\begin{cond}
  \label{cond:pa1}
  \begin{equation}
    \label{eq:pa1}
    \ind {A_i} {\overline U_i} {(\overline L_i, \overline
      A_{i-1}\,;\, \sigma)} 
    \quad(i = 1, \ldots, N).
  \end{equation}
\end{cond}
This is essentially a discrete-time version of Definition~2~(ii) of
\textcite{parner/arjas}, but with the additional vital requirement
that the unobservable variables ${\cal U}$ involved already be such as
to allow us to assume the extended stability property \eqref{es}.
(Without such an underlying assumption there can be no way of relating
different regimes together.)

\condref{pa1} requires that, for each regime, any earlier unobserved
variables in the extended information base ${\cal I}'$ can have no
further effect on the distribution of $A_i$, once the earlier observed
variables are taken into account.  This will certainly be the case
when, under each regime, treatment assignment, at any stage, is
determined by some deterministic or randomizing device that only has
the values of those earlier observed variables as inputs.  While this
will necessarily hold for a control strategy with respect to the
observed information base, whether or not it is a reasonable
requirement for the observational regime will depend on deeper
consideration of the specific context and circumstances.  It will
typically do so if all information available to and utilised by the
decision-maker (the doctor, for instance) in the observational regime
is included in $\overline L_i$, or, indeed, if the actions $(A_i)$
have been physically randomized within levels of $(\overline L_i,
\overline A_{i-1})$.

\begin{theorem}
  \label{thm:parner1}
  Suppose our model exhibits extended stability.  If in addition
  \condref{pa1} holds, then we shall also have the simple stability
  property \eqref{expci2}.
\end{theorem}

\begin{proof}
  Our proof will be based on universal general properties of
  conditional independence, as described by
  \textcite{apd:CIST,apd:enc/ci}.
  
  Let $E_i$, $R_i$, $H_i$ denote, respectively, the following
  assertions:
  \begin{eqnarray*}
    E_i &:&\ind {(L_i,  U_i)} \sigma {(\overline L_{i-1}, \overline U_{i-1}, \overline A_{i-1})}\\
    R_i &:&\ind {A_i} {\overline U_i}{(\overline L_{i}, \overline A_{i-1}; \sigma)}\\
    H_i &:&\ind {(L_i, \overline U_i)} \sigma {(\overline L_{i-1}, \overline A_{i-1})}
  \end{eqnarray*}
  Extended stability is equivalent to $E_i$ holding for all $i$, so we
  assume that; while $R_i$ is just \condref{pa1}, which we are
  likewise assuming for all $i$.  We shall show that these assumptions
  imply $H_i$ for all $i$, which in turn implies $\ind {L_i} \sigma
  {(\overline L_{i-1}, \overline A_{i-1})}$, \ie, simple stability.
  
  We proceed by induction.  Since $E_1$ and $H_1$ are both equivalent
  to $\indo {(L_1, U_1)} \sigma$, $H_1$ holds.
  
  Suppose now $H_i$ holds.  Conditioning on $L_i$ yields
  \begin{equation}
    \label{eq:sr1}
    \ind {\overline U_i} \sigma {(\overline L_{i}, \overline A_{i-1})},
  \end{equation}
  and this together with $R_i$ is equivalent to $\ind {\overline U_i}
  {(A_i, \sigma)} {(\overline L_{i}, \overline A_{i-1})}$, which on
  conditioning on $A_i$ then yields
  \begin{equation}
    \label{eq:sr2}
    \ind {\overline U_i} \sigma {(\overline L_{i}, \overline A_{i})}.
  \end{equation}
  Also, by $E_{i+1}$ we have
  \begin{equation}
    \label{eq:sr3}
    \ind {(L_{i+1},  U_{i+1})} \sigma {(\overline L_{i}, \overline U_{i}, \overline A_{i})}.
  \end{equation}
  Taken together, \eqref{sr2} and \eqref{sr3} are equivalent to
  $H_{i+1}$, so the induction is established.
\end{proof}

\subsubsection{Sequential irrelevance}

Another possible condition is:
\begin{cond}
  \label{cond:pa1a}
  \begin{equation}
    \label{eq:pa1a}
    \ind {L_i} {\overline U_{i-1}} {(\overline L_{i-1}, \overline
      A_{i-1}\,;\, \sigma)} \quad (i = 1, \ldots, N+1).
  \end{equation}
\end{cond}
In contrast to \eqref{pa1}, \eqref{pa1a} does permit the unobserved
variables to date, $\overline U_{i}$, to influence the next action
$A_i$ (which can however only happen in the observational regime), as
well as the current observable $L_i$; but they do not affect the
subsequent development of the $L$'s (including, in particular, the
response variable $Y$).

\begin{theorem}
  \label{thm:seqirr}
  Suppose:
  \begin{enumerate}
\item \label{it:es} Extended stability, \eqref{es}, holds.
\item \label{it:si} Sequential irrelevance, \condref{pa1a}, holds for
      the observational regime $\sigma = \co$:
    \begin{equation}
      \label{eq:si}
      \ind {L_i} {\overline{U}_{i-1}} {(\overline{L}_{i-1}, \overline{A}_{i-1};\sigma = \co)}
      \quad (i = 1, \ldots, N+1).
    \end{equation}
\item \label{it:pos} Extended positivity, as in \defref{extpositive},
      holds.
  \end{enumerate}
  Then we shall have simple stability:
  \begin{equation}
    \label{eq:stab}
    \ind {L_i} \sigma  {(\overline{L}_{i-1}, \overline{A}_{i-1})}\quad (i = 1, \ldots, N+1).
  \end{equation}
  Moreover, sequential irrelevance holds under any regime:
  \begin{equation}
    \label{eq:si2}
    \ind {L_i} {\overline{U}_{i-1}} {(\overline{L}_{i-1}, \overline{A}_{i-1};\sigma)}\quad (i = 1, \ldots, N+1).
  \end{equation}
  
\end{theorem}

\begin{proof}
  Let $k(L_i)$ be a bounded real function of $L_i$, and, for each
  regime $s \in \cs$, let $h(\overline{U}_{i-1}, \overline{L}_{i-1},
  \overline{A}_{i-1}; s)$ be a version of $\E \{k(L_i) \mid
  \overline{U}_{i-1}, \overline{L}_{i-1}, \overline{A}_{i-1}; s\}$.
  
  By \eqref{si} there exists $f(\overline{L}_{i-1},
  \overline{A}_{i-1})$ such that
  \begin{equation}
    \label{eq:aa1}
    h(\overline{U}_{i-1}, \overline{L}_{i-1},
    \overline{A}_{i-1};\co) = f(\overline{L}_{i-1},
    \overline{A}_{i-1})\quad\as{P_\co}
  \end{equation}
  whence, from \eqref{pos}, for all $s \in \cs$,
  \begin{equation}
    \label{eq:aa2}
    h(\overline{U}_{i-1}, \overline{L}_{i-1},
    \overline{A}_{i-1};\co) = f(\overline{L}_{i-1},
    \overline{A}_{i-1})\quad\as{P_s}.
  \end{equation}
  
  Also, from \eqref{es},
  \begin{equation}
    \label{eq:es0}
    \ind {L_i} \sigma {\overline{U}_{i-1}, \overline{L}_{i-1}, \overline{A}_{i-1}}
  \end{equation}
  and so there exists $g(\overline{U}_{i-1}, \overline{L}_{i-1},
  \overline{A}_{i-1})$ such that, for all $s \in \cs$,
  \begin{equation}
    \label{eq:aa3}
    h(\overline{U}_{i-1}, \overline{L}_{i-1},
    \overline{A}_{i-1}; s) = g(\overline{U}_{i-1},
    \overline{L}_{i-1}, \overline{A}_{i-1})
    \quad\as{P_s}.
  \end{equation}
  In particular,
  \begin{equation}
    \label{eq:aa3a}
    h(\overline{U}_{i-1}, \overline{L}_{i-1},
    \overline{A}_{i-1};\co) = g(\overline{U}_{i-1},
    \overline{L}_{i-1}, \overline{A}_{i-1})
    \quad\as{P_\co},
  \end{equation}
  so that, again using \eqref{pos},
  \begin{equation}
    \label{eq:aa4}
    h(\overline{U}_{i-1}, \overline{L}_{i-1},
    \overline{A}_{i-1};\co) = g(\overline{U}_{i-1},
    \overline{L}_{i-1}, \overline{A}_{i-1})
    \quad\as{P_s}.
  \end{equation}
  Combining \eqref{aa2}, \eqref{aa3} and \eqref{aa4}, we obtain
  \begin{equation}
    \label{eq:aa5}
    h(\overline{U}_{i-1}, \overline{L}_{i-1},
    \overline{A}_{i-1}; s) =  f(\overline{L}_{i-1},
    \overline{A}_{i-1}) \quad\as{P_s}.
  \end{equation}
  Since this property holds for all $s \in \cs$ and every bounded real
  function $k(L_i)$, we deduce
  \begin{equation}
    \label{eq:stab0}
    \ind {L_i} {(\overline{U}_{i-1}, \sigma)}  {(\overline{L}_{i-1}, \overline{A}_{i-1})}
  \end{equation}
  from which both \eqref{stab} and \eqref{si2} follow.
\end{proof}

It is worth noting that we do not need the full force of extended
stability for the above proof, but only \eqref{es0}.  In particular,
we could allow arbitrary dependence of $U_i$ on any earlier variables,
including $\sigma$.  We note further that the above proof makes
essential use of the extended positivity property of
\defref{extpositive}: \eqref{stab} can not be deduced from extended
stability and \condref{pa1a} making use of the standard conditional
independence axioms \cite{apd:enc/ci,pearl/paz,apd:separoids} alone.

Although we can certainly deduce simple stability when we can assume
the conditions of either \thmref{parner1} or \thmref{seqirr}, it can
also arise our of extended stability in other ways.  For example, this
can be so when \condref{pa1} holds for some subsets of $\overline
U_i$, while \condref{pa1a} holds for some subsets of $\overline
U_{i-1}$.  Such cases are addressed by Corollaries~4.1 and 4.2 of
\textcite{jmr:lns}; we give examples in \secref{furtherex} below.

\section{Influence diagrams}
\label{sec:ID}

As previously mentioned, it is often helpful (though never essential)
to represent and manipulate conditional independence properties
graphically, using the formalism of {\em influence diagrams\/} (IDs).
In particular, when including unobserved variables $\cal U$ and
assuming extended stability, we can often deduce directly from
graph-theoretic separation properties whether simple stability holds.

\subsection{Semantics}
\label{sec:semantics}
Here we very briefly describe the semantics of IDs, and show how they
can facilitate logical arguments by displaying implied properties in a
particularly transparent form.  We shall use the theory and notation
of \textcite{mybook} and \textcite{apd:infdiags} in relation to
directed acyclic graphs (DAGs) and IDs, and their application to
probability and decision models.  The reader is referred to these
sources for more details.

For any DAG or ID $\dag$, its {\em moral graph\/}, or {\em
  moralization\/}, is the undirected graph $\mo(\dag)$ in which first
an edge is inserted between any unlinked parents of a common child in
$\dag$, and then all directions are ignored.  For any set $S$ of nodes
of $\dag$ we denote the smallest ancestral subgraph of $\dag$
containing $S$ by $\an_\dag(S)$, and its moralization by $\ma_\dag(S)$
(we may omit the specification of $\dag$ when this is clear).  For
sets $A, B, C$ of nodes of $\dag$ we write $\inda A B C \dag$, and say
{\em $C$ separates $A$ from $B$ (with respect to $\dag$)\/} to mean
that, in $\ma(A \cup B \cup C)$, every path joining $A$ to $B$
intersects $C$. Let $\nd(V)$ and $\parents{V}$ denote the
non-descendants and parents of a random node $V$, then it can be shown
\cite{ldll,apd:infdiags} that, whenever a probability distribution or
decision problem is {\em represented by $\dag$\/}, in the sense that
for any such $V$ the probabilistic conditional independence $\ind V
{\nd(V)} {\parents{V}}$ holds, we have
\begin{equation}
  \label{eq:moralcrit}
  \inda A B C \dag \Rightarrow \ind A B C.
\end{equation}
This {\em moralization criterion\/} thus allows us to infer
probabilistic independence properties from purely graph-theoretic
separation properties.%
\footnote{%
  An alternative, and entirely equivalent, approach can be based on
  the `$d$-separation criterion' \cite{verma,pearl:book}.  We have
  found \eqref{moralcrit} more straightforward to understand and
  apply.%
}

While the above allows us to read off conditional independencies from
a DAG, we can, conversely, construct an ID $\dag$ from a given
collection of joint distributions over the domain variables (one for
each regime) in the following way.

The node-set is given by ${\cal V} = \{\sigma\} \cup {\cal L} \cup
{\cal U} \cup {\cal A}$.  The graph has random (round) nodes for all
the domain variables, and a founder decision (square) node for
$\sigma$.  The ordering given by the extended information base ${\cal
  I}'$ induces an ordering on ${\cal V}$ such that any nodes in the
(possibly empty) sets $L_i$, $U_i$ come after $A_{i-1}$ and before
$A_{i}$, and $L_{N+1}\equiv Y$ is last.  In addition we require the
node $\sigma$ to be prior to any domain variables in this ordering.
With each node $\nu\in {\cal V}_0 := {\cal V}\setminus\{\sigma\}$ is
associated its collection of conditional distributions, given values
for all its {\em predecessors\/}, $\pre{\nu}$, in the ordering
(including, in particular, specification of the relevant regime).

For each such $\nu$ we will have a conditional independence (CI)
property of the form:

\begin{displaymath}
  C(\nu) : \ind \nu {\pre{\nu}} {\parents{\nu}}
\end{displaymath}
where $\parents{\nu}$ is some given subset of $\pre{\nu}$.  Thus
$C(\nu)$ asserts that the distributions of $\nu$, given all its
predecessors, in fact only depends on the values of those in
$\parents{\nu}$.  Note that property $C(\nu)$ will be vacuous, and can
be omitted, when $\parents{\nu} = \pre{\nu}$.  Such a collection,
${\cal C}$ say, of CI properties is termed {\em recursive\/}.  We
represent ${\cal C}$ graphically by drawing an arrow into each node
$\nu\in{\cal V}_0$ from each member of its parent set $\parents{\nu}$,
and we associate with $\nu$ the `parent-child' conditional
probabilities of the form $p(\nu = \nu^* \mid\parents{\nu} = pa^*)$.
The ID constructed in this way will ensure that the joint distribution
of the domain variables, in each regime, satisfies any conditional
independencies obtained by applying the moralization criterion
\eqref{moralcrit}.

From this point on, when we use the terms `parents', `ancestors' \etc,
the regime node $\sigma$ will be excluded from these sets.  Also,
while in general the terms $L_i$, $U_i$ could each refer to a
collection of variables, for simplicity we shall consider only the
case in which they represent just one (or sometimes none), and so can
be modelled (if present at all) by a single node in the graph.

We emphasise that IDs are related to but distinct from `causal DAGs'
\cite{sgs:book,pearl:bka}.  For a discussion see \textcite{apd:beware}
and \textcite{didetal:10}.

\subsection{Extended stability}
\label{sec:extstabdag}

The extended stability property \eqref{es} embodies a recursive
collection of CI properties with respect to the ordering induced by
the extended information base.  Consequently it can be faithfully
expressed by an ID $\dag$ satisfying:

\begin{cond}
  \label{cond:sigout}
  The only arrows out of $\sigma$ in $\dag$ are into ${\cal A}$.
\end{cond}

\begin{figure}[h]
  \begin{center}
    \resizebox{2,2in}{!}{\includegraphics{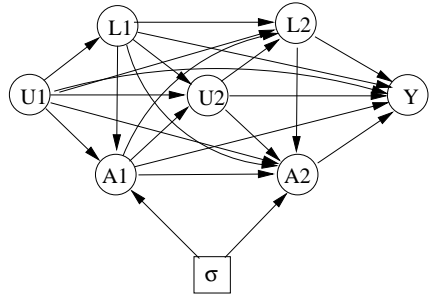}}
    \caption{Unobserved variables: $N=2$}
    \label{fig:parner}
  \end{center}
\end{figure}

\noindent For $N=2$ this is depicted in \figref{parner}.  Note that
the subgraph corresponding to the domain variables is complete.

\subsubsection{Sequential randomization }
\label{sec:seqrandgraph}

With the ordering induced by the extended information base ${\cal
  I}'$, \eqref{es} and \eqref{pa1} together form a recursive
collection ${\cal C}$ of CI properties.  Therefore the conditions of
\thmref{parner1} can be faithfully represented graphically in an ID
$\dag$, in which, for extended stability, the only arrows out of
$\sigma$ are into the $A$'s, while also, for sequential randomization,
there are no arrows into the $A$'s from the $U$'s.  Thus starting from
\figref{parner}, for example, we simply delete all the arrows from a
$U$ to an $A$, so obtaining \figref{parner1}.

\begin{figure}[h]
  \begin{center}
    \resizebox{2.2in}{!}{\includegraphics{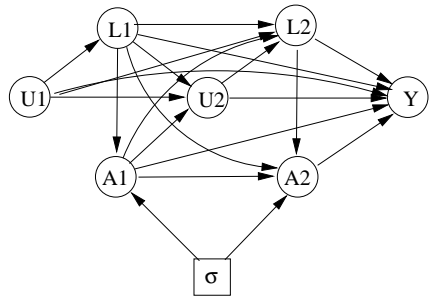}}
    \caption{ID showing sequential randomization.}
    \label{fig:parner1}
  \end{center}
\end{figure}

We can now verify \thmref{parner1} using only graphical manipulations,
as follows.

Since, under \eqref{es}, the only children of $\sigma$ are action
variables, and under \eqref{pa1} no action variable can be a child of
any unobservable variable, it follows that in $\ma(\sigma, \overline
L_{i}, \overline A_{i-1})$ there will be no direct link between
$\sigma$ and any $U\in{\cal U}$.  A similar argument shows that
\eqref{es} implies that there is no direct link in $\ma(\sigma,
\overline L_{i}, \overline A_{i-1})$ between $\sigma$ and $L_i$.  It
follows that every path from $L_i$ to $\sigma$ must pass through one
of the remaining variables, \ie\ $(\overline L_{i-1}, \overline
A_{i-1})$, demonstrating that $ \inda {L_i} {\sigma} {(\overline
  L_{i-1}, \overline A_{i-1})} \dag$ for $i = 1, \ldots, N+1$.  Simple
stability \eqref{expci2} now follows from \eqref{moralcrit}.

\subsubsection{Sequential irrelevance }
\label{sec:seqirrgraph}

The case of sequential irrelevance is more subtle.  This is because
when we combine extended stability \eqref{es} with sequential
irrelevance \eqref{pa1a} we do not obtain a recursive collection of CI
properties.  Consequently this combined collection of conditional
independencies cannot be faithfully represented by any ID.

It might be thought that, starting with an ID representing extended
stability, we could operate on it to incorporate \eqref{pa1a} also
simply by deleting all arrows from $U_i$ into $L_j$ for $j>i$.  Doing
this to \figref{parner} yields the ID of \figref{parner1a}.  However,
that ID also represents the stronger property \eqref{stab0} (shown by
the absence of edges from $\sigma$ and $\overline U_{i-1}$ into
$L_i$), which does not follow from \eqref{es} and \eqref{pa1a} without
imposing further, non-graphical conditions (as was done in
\thmref{seqirr}).  We can indeed read off the stability property
\eqref{expci2} from \figref{parner1a}, but while that graph thus
displays clearly the conclusion of \thmref{seqirr}, it does not supply
an alternative graphical proof.

\begin{figure}[h]
  \begin{center}
    \resizebox{2.2in}{!}{\includegraphics{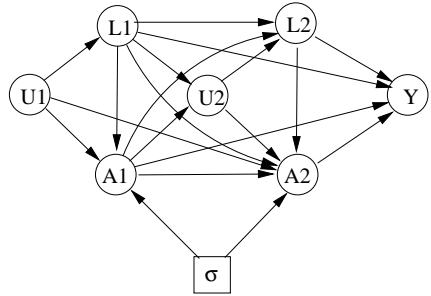}}
    \caption{ID implying sequential irrelevance.}
    \label{fig:parner1a}
  \end{center}
\end{figure}

By omitting some of the nodes and/or arrows in an ID, such as
\figref{parner1} or \figref{parner1a}, that already embodies either
sequential randomization or sequential irrelevance, we obtain simpler
special cases with the same property.  Two such examples, starting
from \figref{parner1a}, are given in \figref{vanessa}.

\begin{figure}[ht]
  \begin{centering}
    \epsfxsize=1.7in \epsfbox{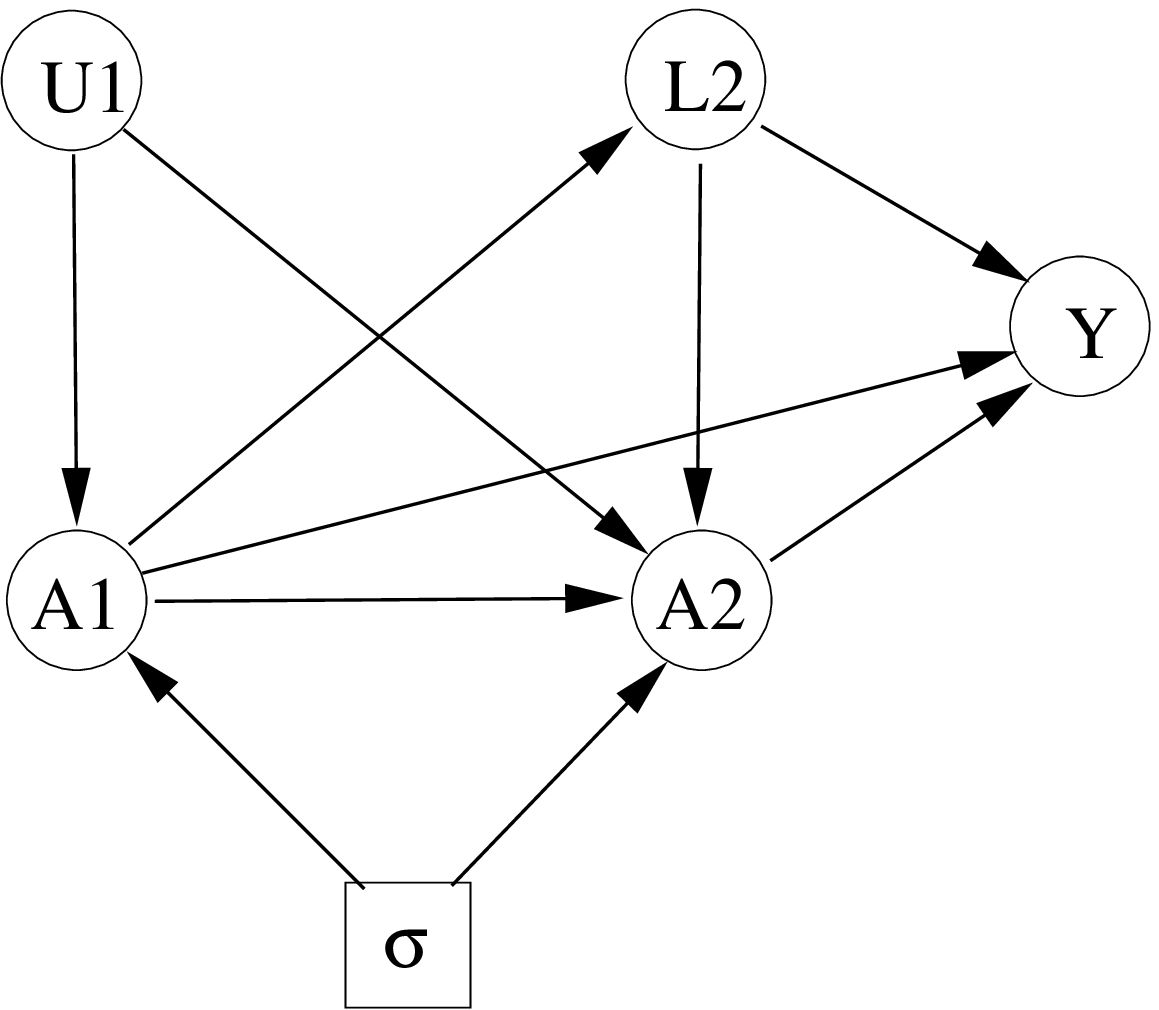} \quad\quad\quad
    \epsfxsize=1.7in \epsfbox{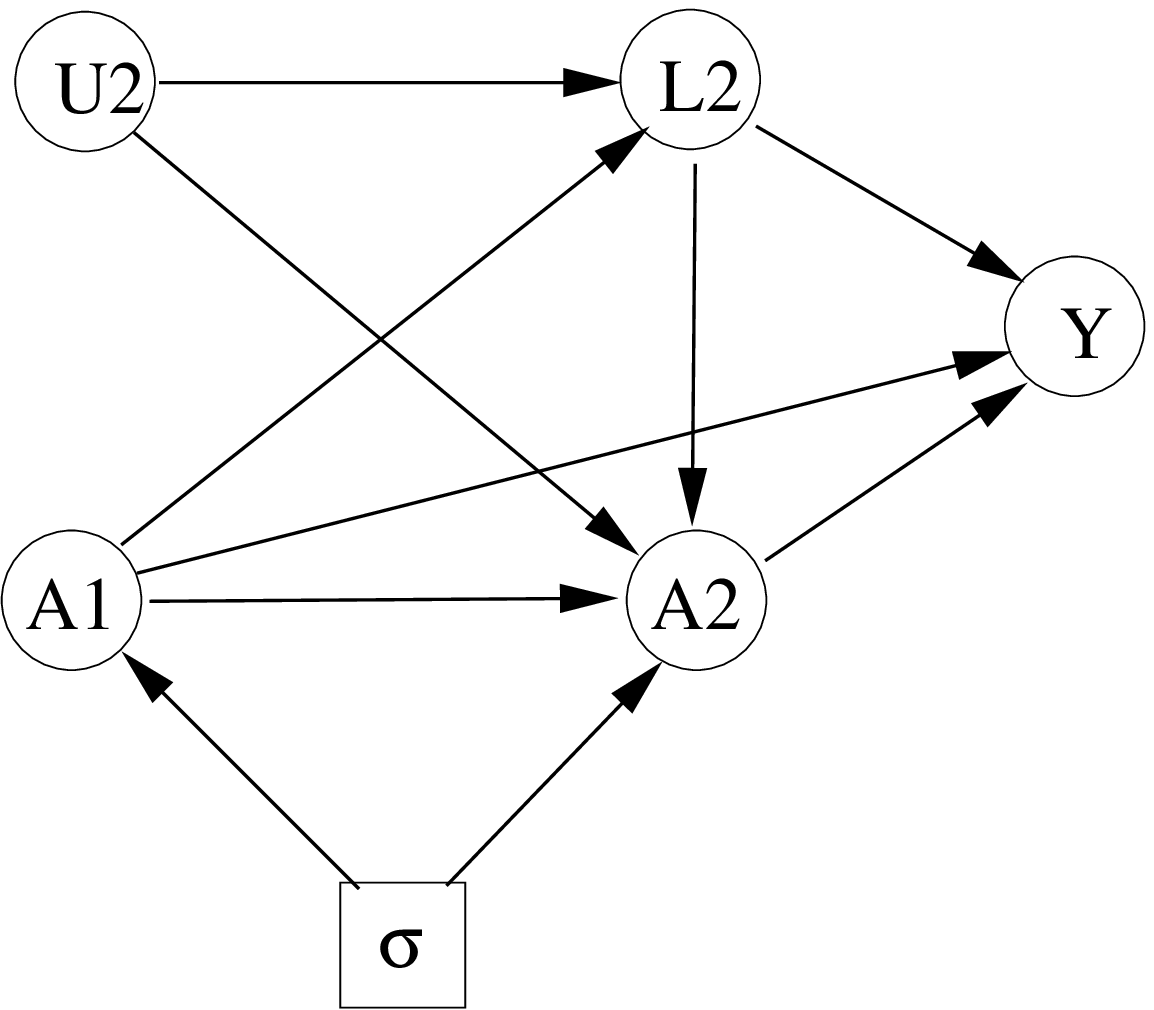}
    \caption{Specialisations of \figref{parner1a}}
    \label{fig:vanessa}
  \end{centering}
\end{figure}

\subsubsection{Further examples}
\label{sec:furtherex}

As mentioned before, we can have simple stability even when both
sequential randomization and sequential irrelevance (or more
precisely, the conditions of Theorems~\ref{thm:parner1} and
\ref{thm:seqirr}) fail.  Two examples are given by the IDs of
\figref{pax}.  Applying the moralisation criterion to the graphs, we
verify, for example, that in both IDs of \figref{pax} simple stability
is satisfied.

\begin{figure}[ht]
  \begin{centering}
    \epsfxsize=2.1in \epsfbox{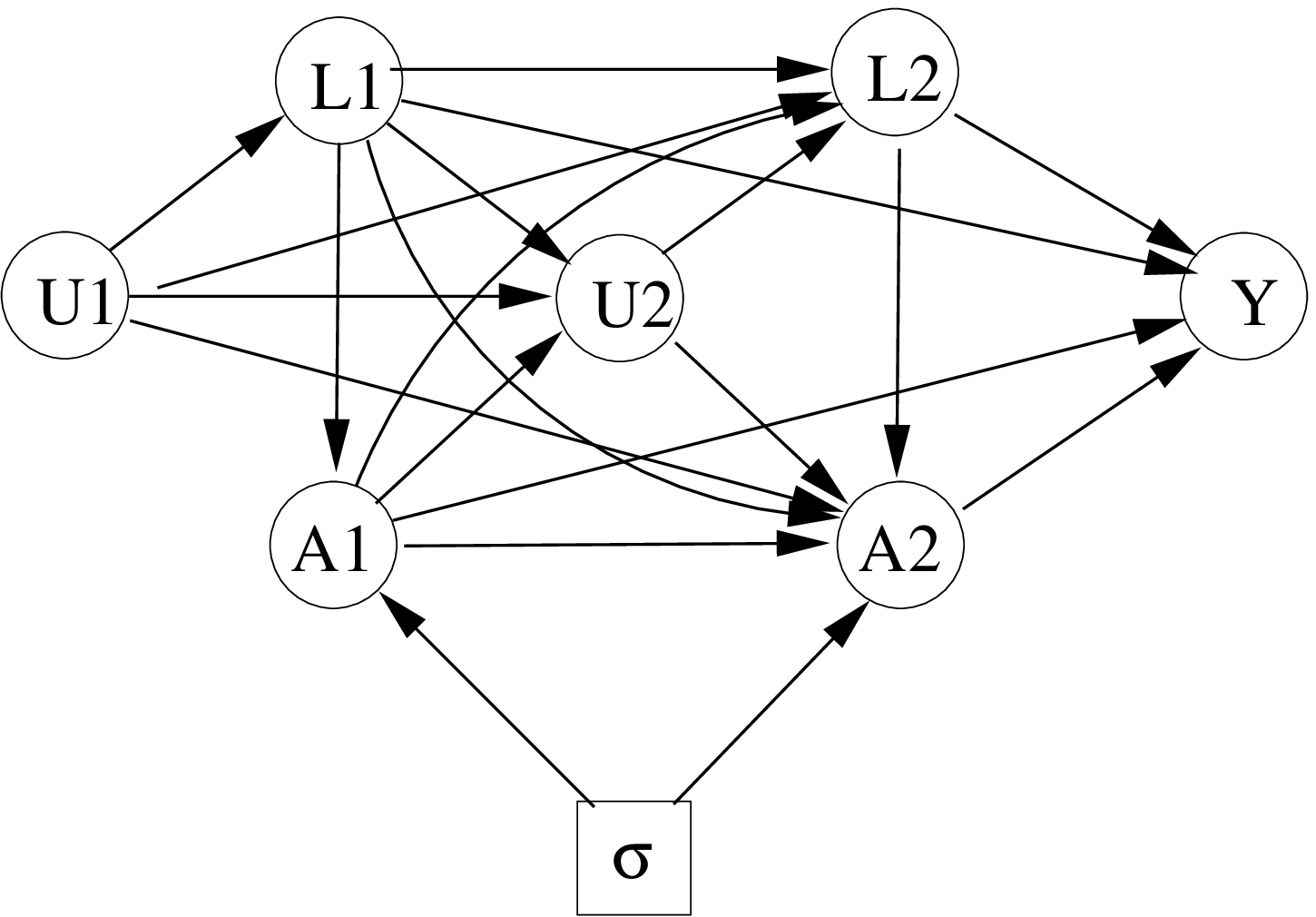}
    \quad\quad\quad
    \epsfxsize=2.1in \epsfbox{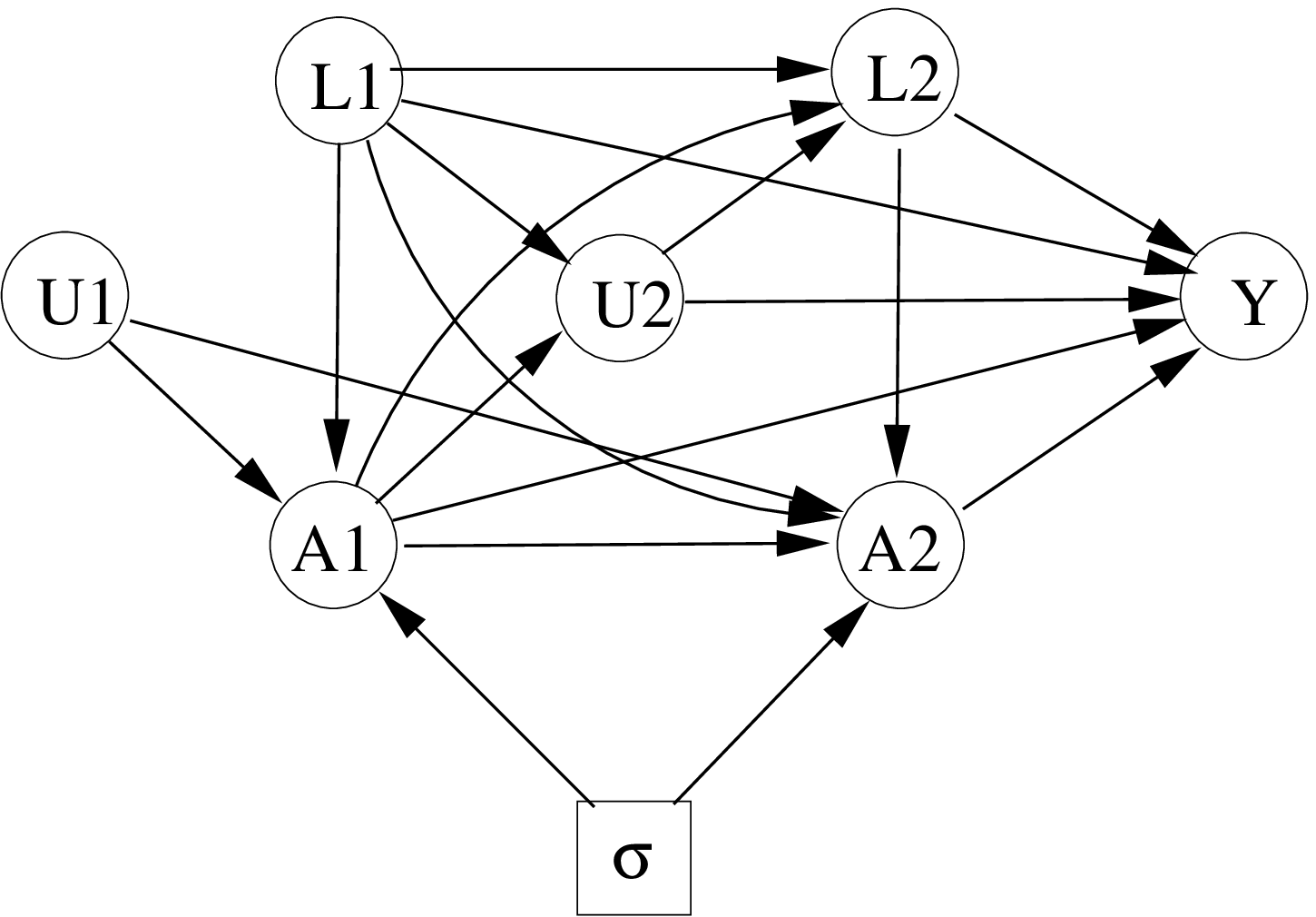}
    \caption{Alternative IDs displaying stability}
    \label{fig:pax}
  \end{centering}
\end{figure}
 
In full generality it is easy to see, using \condref{sigout}, that
application of the moralization criterion to $\dag$ to check the
simple stability condition \eqref{expci2} is equivalent to checking
that, for each $i$, $\overline{L}_{i-1}$ satisfies Pearl's {\em
  back-door criterion\/} \cite{pearl:bka} relative to
$(\overline{A}_{i-1}, L_i)$.  (Pearl only considers atomic
interventions, but our analysis shows that this condition also allows
identification of conditional interventions.)

\subsubsection{Positivity}
\label{sec:grecgen}
Suppose that (whether by appealing to sequential randomization, or to
sequential irrelevance, or the back-door criterion, or otherwise) we
have been able to demonstrate simple stability with respect to an
observable information base.  Suppose further that $\ce$ is a control
strategy in the sense of \condref{cont}.  It will now follow that we
can use $G$-recursion, exactly as in \secref{grec0}, to identify the
consequence of regime $\ce$ from data gathered under regime $\co$ ---
so long only as we can also ensure the positivity constraint of
\defref{positive}.

It is easy to see that a sufficient condition for \defref{positive} to
hold is:

\begin{cond}[Parent-child positivity]
  \label{cond:pos}
  For each $A\in{\cal A}$, and each configuration $(a, pa^*)$ of $(A,
  \pa{\dag}{A})$, $p(a \mid pa^*; \ce) >0 \Rightarrow p(a \mid pa^*;
  \co) >0$.
\end{cond}

More generally, suppose that we specify, for each entry in each
parent-child conditional probability table for the ID $\dag$, whether
it is zero or non-zero.  We can then apply {\em constraint
  propagation\/} algorithms \cite{dechter:03} to determine ${\cal E}$
and ${\cal O}$.  One such method \cite{dawid:92} uses an analogue of
the computational method of probability propagation \cite{mybook}.
This generates a collection of `cliques' (subsets of the variables)
with, for each clique, an assignment of 1 (meaning possible) or 0
(impossible) to each configuration of its variables.
\defref{positive} will then hold if and only if, for each clique
containing $\sigma$, no entry changes from 0 to 1 when we change the
value of $\sigma$ from $\co$ to $\ce$.

\section{A more general approach}
\label{sec:moregen}

The simple stability condition \eqref{expci2} requires that, for each
$i$, the conditional distribution of $L_i$, given the earlier
variables $(\overline L_{i-1}, \overline A_{i-1})$, should be the same
under both regimes $\co$ and $\ce$ --- a strong assumption that, in
certain problems, one might be unwilling to accept directly, and
unable to deduce, as in \secref{regained}, from more acceptable
assumptions.  However, while we have shown that stability (together
with \defref{positive}) is sufficient to support $G$-recursion, it
turns out not to be necessary.

In this section we first give some very general conditions under which
$G$-recursion can be justified; then we consider their specific
application to models incorporating extended stability.  Our analysis
parallels parts of \textcite{jr:cma} (see also Section 3.4 of
\textcite{jmr:lns}), in which the `sequential randomization'
assumption is relaxed.  We consider the relation between the two
approaches in more detail in \secref{judith}.

Rather than work directly with \eqref{standrecura} and
\eqref{standrecurl}, we combine them into the following form:
\begin{equation}
  \label{eq:standrecur}
  f(\overline l_{i-1}, \overline a_{i-1}) = \sum_{l_i}\sum_{a_i}\,
  p(l_i \mid \overline l_{i-1}, \overline a_{i-1}; \co) \times p(a_i
  \mid \overline l_{i}, \overline a_{i-1}\,;\, \ce) \times f(
  \overline l_i, \overline a_i).
\end{equation}
To justify $G$-recursion it is enough to demonstrate the applicability
of \eqref{standrecur}.

\subsection{$G$-recursion: General conditions}
\label{sec:grec}
A primitive building block of our model is the specification of the
interventional conditional probabilities $p(a_i \mid \overline l_i,
\overline a_{i-1}\,;\,\ce)$.  We suppose that this is well-defined
(\eg\ by deterministic functions or specified randomization) at least
for all $(\overline l_i, \overline a_{i-1}) \in{\cal O}$ ($1 \leq i
\leq N$), even if $(\overline l_i, \overline a_{i-1}) \not\in{\cal
  E}$.

We introduce a function $\gamma: {\cal H} \rightarrow \{0, 1\}$
defined by:
\begin{equation}
  \label{eq:gamma}
  \gamma(h) := 
  \left\{
    \begin{array}[c]{ll}
      1 & \mbox{if } h \in{\cal O} \mbox{ and }
      \prod_{j=1}^{i} p(a_j \mid  \overline l_{j}, \overline a_{j-1}\,;\,
      \ce) > 0\\
      0 & 
      \mbox{otherwise.}
    \end{array}
  \right.
\end{equation}
In \eqref{gamma}, $i$ is the highest index of an action variable
appearing in $h$, \ie $h = (\overline l_i, \overline a_i)$ or
$(\overline l_{i+1}, \overline a_i)$.  Note that if $h$ is an initial
segment of $h'$, then $\gamma(h) = 0\Rightarrow \gamma(h') = 0$.

We define:
\begin{equation}
  \label{eq:Gamma}
  \Gamma := \{h\in {\cal H} : \gamma(h) = 1\}
\end{equation}
(so that, in particular, $\Gamma \subseteq {\cal O}$).

We now impose the following positivity condition in place of
\defref{positive}:
\begin{cond}
  \label{cond:ac2}
  For $1 \leq i \leq N$, if $(\overline l_{i}, \overline a_{i-1})$ is
  in ${\Gamma}$ and $p(a_i \mid \overline l_{i}, \overline
  a_{i-1}\,;\,\ce) > 0$, then $(\overline l_i, \overline a_i)$ is in
  ${\cal O}$ (and thus in ${\Gamma}$).
\end{cond}
This requires that, subsequent to any partial history $(\overline l_i,
\overline a_{i-1})$ in ${\Gamma}$, if some value of the next action
variable can be generated by intervention, it can also arise
observationally.

Our approach now involves the construction, if possible, of a sequence
of joint distributions $p_i(\,\cdot\,)$ ($i = 0, \ldots, N$) for all
the variables in the problem, such that
\begin{equation}
  \label{eq:0=e}
  p_0(y) \equiv p(y \,;\, \ce),
\end{equation}
and certain further properties hold, as described below.  For maximum
applicability these are stated here in a very abstract and general
form.  Some concrete cases where we can specify suitable $(p_i)$ and
verify that they have the requisite properties are treated in
\secref{extstab} and \secref{judith} below.

Let the class of partial histories $h\in{\cal H}$ having positive
probability under $p_i$ be denoted by ${\cal B}_i$, and let
${\Gamma}_i := {\cal B}_i \cap {\Gamma}$.

We require the following positivity property:
\begin{equation}
  \label{eq:acnew}
  (\overline l_{i}, \overline a_i) \in{\cal B}_i \Leftrightarrow 
  (\overline l_{i}, \overline a_i) \in{\cal O}.
\end{equation}

Since $\Gamma\subseteq {\cal O}$, from `$\Leftarrow$' in \eqref{acnew}
we readily deduce
\begin{equation}
  \label{eq:acnew2}
  (\overline l_{i}, \overline a_i) \in \Gamma \Leftrightarrow  (\overline
  l_{i}, \overline a_i) \in \Gamma_i.
\end{equation}

More substantively we require:
\begin{eqnarray}
  \label{eq:N=o}
  \label{eq:l}
  p_{i-1}(l_i \mid \overline l_{i-1}, \overline a_{i-1}) &=& p(l_i \mid \overline l_{i-1},
  \overline a_{i-1}\,;\, \co)\quad(i=1, \ldots, N+1)\\
  \label{eq:a}
  p_{i-1}(a_i \mid \overline l_i, \overline a_{i-1}) &=& p(a_i \mid \overline l_i, \overline a_{i-1}\,;\,
  \ce)\quad(i=1,\ldots, N+1)\\
  \label{eq:y}
  p_{i-1}(y \mid \overline l_i, \overline a_i) &=& p_{i}(y \mid
  \overline l_i, \overline a_i)\quad(i=1, \ldots, N)
\end{eqnarray}
{\em whenever, in each case, the conditioning partial history on the
  left-hand side is in ${\Gamma}_{i-1}$\/} (in which case the
conditional probabilities on both sides are unambiguously defined).

Suppose now that such a collection of distributions $(p_i)$ can be
found.  Let ${\cal H}_0$ denote the set of all partial histories of
the form $(\overline l_i, \overline a_i)$ for some $i$.  We define a
function $f: {\cal H}_0 \rightarrow \Re$ by:
\begin{equation}
  \label{eq:fi}
  f(h) :=  \gamma(h) \times \E_{i}\{k(Y) \mid h\},
\end{equation}
for $h = (\overline l_i, \overline a_i)$, where $\E_i$ denotes
expectation under $p_i$.  We note that $f$ is well-defined, since
$\gamma(h) \neq 0 \Rightarrow h\in{\cal O}$, whence $h\in{\cal B}_i$
by \eqref{acnew}.

For $h = (\overline l_N, \overline a_N)$, if $\gamma(h) \neq 0$ then
by \eqref{acnew2} $h \in\Gamma_N$, so that we can apply \eqref{l} for
$i = N+1$ to see that:
\begin{equation}
  \label{eq:fngen}
  f(\overline l_N, \overline a_N) = 
  \left\{
    \begin{array}[c]{ll}
      \E\{k(Y) \mid \overline l_N, \overline a_N\,;\,\co\} &\mbox{if }
      (\overline l_N, \overline a_N)\in\Gamma\\
      0 & \mbox{otherwise.}
    \end{array}
  \right.
\end{equation}
Also, by \eqref{0=e},
\begin{equation}
  \label{eq:fempty}
  f(\emptyset) = p(y\,;\, \ce).  
\end{equation}

\begin{lem}
  \label{lem:induction}
  Under \condref{ac2} and properties~\eqref{0=e}--\eqref{y}, the
  $G$-recursion \eqref{standrecur} holds for the interpretation
  \eqref{fi}.
\end{lem}

\begin{proof}
  If $\gamma(\overline l_{i-1}, \overline a_{i-1}) = 0$ then both
  sides of \eqref{standrecur} are 0.
  
  Otherwise $(\overline l_{i-1}, \overline a_{i-1})$ is in $\Gamma$
  and so, by \eqref{acnew2}, in $\Gamma_{i-1}$.  We have:
  \begin{eqnarray}
    \nonumber
    f(\overline{l}_{i-1}, \overline{a}_{i-1}) &=& \E_{i-1}\{k(Y) \mid
    \overline{l}_{i-1}, \overline{a}_{i-1}\}\\
    \nonumber
    &=& \sum_{l_i} \, \sum_{a_i} \,\,
    p_{i-1}(l_i \mid \overline{l}_{i-1}, \overline{a}_{i-1}) \times  p_{i-1}(a_i
    \mid \overline{l}_i, \overline{a}_{i-1}) \times  \E_{i-1}\{k(Y) \mid \overline{l}_i,
    \overline{a}_i\}.\\
    \label{eq:inta}\mbox{\hspace{1em}}
  \end{eqnarray}
  Denote the three terms on the right-hand side of \eqref{inta} by
  $T_l$, $T_a$, $T_y$, respectively.  By \eqref{l} $T_l= p(l_i \mid
  \overline{l}_{i-1}, \overline{a}_{i-1}\,;\,\co)$.  We do not need to
  consider the other terms when $T_l$ is 0.  Otherwise,
  $(\overline{l}_i, \overline{a}_{i-1})$ is in $\Gamma_{i-1}$.  By
  \eqref{a}, we now have $T_a = p(a_i \mid \overline{l}_i,
  \overline{a}_{i-1}\,;\,\ce)$.  Again we do not have to worry about
  $T_y$ unless $T_a $ is non-zero.  In that case $(\overline{l}_i,
  \overline{a}_i)$ is in ${\cal B}_{i-1}$ and also, by \condref{ac2},
  in $\Gamma$, hence in $\Gamma_{i-1}$.  We can now use \eqref{y} to
  replace $T_y$ by $\E_i\{k(Y) \mid \overline{l}_i, \overline{a}_i\} =
  f(\overline{l}_i, \overline{a}_i)$, and the result follows.
\end{proof}

Starting from \eqref{fngen}, we can thus apply $G$-recursion as given
by \eqref{standrecur}, or equivalently by \eqref{standrecura} and
\eqref{standrecurl}, to compute $f(\emptyset)$ --- which, by
\eqref{fempty}, is just the desired consequence of regime $\ce$.  In
this computation we only need consider partial histories in $\Gamma$.
When $\ce$ is a deterministic strategy we recover the form
\eqref{determstandrecur} of $G$-recursion.

Note that, for histories of intermediate length, the function $f$
defined by \eqref{fi} involves the constructed distributions $(p_i)$,
which need not have any real-world interpretation.  Note further that,
in contrast to the case when stability applies, even when we can use
the above construction to compute the marginal interventional
distribution of the response variable $Y$, there is no guarantee that
we can identify the full joint interventional distribution of
$(\overline L_N, \overline A_N, Y)$.  In particular, if the loss
function depends on variables other than $Y$ we may not be able to
estimate the expected loss of an interventional strategy on the basis
of observational data.

\subsection{Extended stability}
\label{sec:extstab}
We now specialize the general approach of \secref{grec} to problems
exhibiting extended stability, as in \eqref{es}.  This can be regarded
as extending the analysis of \textcite{pearl/robins:uai95} to handle
dynamic regimes, as also considered by
\textcite{jmr:lns}.\footnote{Both these papers refer for the details
  to an unpublished paper, Robins and Pearl (1996).}

We aim to identify a graphical counterpart to the conditions of
\secref{grec}, that would allow us to apply $G$-recursion to this
extended information base so as to identify the effect of regime $\ce$
from observations made under $\co$.

For the remainder of this section we consider a given information base
${\cal I}'$ that induces an ordering of the nodes of the influence
diagram $\dag$; in \secref{vdproof} we consider the converse, \ie\ how
to find an ordering of the information base from a given influence
diagram $\dag$ such that the graphical check of \secref{check}
succeeds.

We impose \condref{pos}.  It is then easy to see that \condref{ac2}
will hold (and in fact $\Gamma = {\cal E}$).  We also impose
\condref{cont} on the control strategy $\ce$.

For each $i=0, \ldots, N$, we now construct an artificial joint
distribution $p_i(\,\cdot\,)$ for all the domain variables as follows.
The distribution $p_i$ factors according to the ID $\dag' = \dag$ with
the node $\sigma$ removed.  The parent-child tables for any variable
$V \in {\cal L}\cup{\cal U}$ are unchanged from the original ones for
$\dag$ (which do not involve $\sigma$).  That for any action variable
$A_j$ for $j \leq i$ is the same as for $\dag$, conditional on $\sigma
= \co$; while that for $A_j$ ($j > i$) is the same as for $\dag$,
conditional on $\sigma = \ce$.

With this definition, $p_0(\,\cdot\,) \equiv p(\,\cdot\,;\,\ce)$, so
that \eqref{0=e} holds.  Properties~\eqref{acnew}, and \eqref{l} for
$i \leq N$, hold because the joint distribution of all variables up to
and including $L_i$ is the same under $p_{i-1}$ as under
$p(\,\cdot\,;\,\co)$; for \eqref{l} with $i = N+1$, when $L_{N+1}
\equiv Y$, we also use the fact that extended stability, \ie\ 
\condref{sigout}, implies that the distribution of $Y$ given all
earlier domain variables is the same under both $\ce$ and $\co$.

Finally \eqref{a} holds because, by construction, the parent-child
distribution for $A_i$ has the same specification for $p_{i-1}(\cdot)$
as for $p(\,\cdot\,;\ce)$ --- and, by \condref{cont}, $\parents{A_i}
\subseteq (\overline L_i, \overline A_{i-1})$.

\subsubsection{Graphical check}
\label{sec:check}
We have shown that, under Conditions~\ref{cond:cont} and
\ref{cond:pos}, properties \eqref{0=e}--\eqref{a} hold automatically
for our above construction of $(p_i)$.  However, whether or not
\eqref{y} holds will depend on more specific conditional independence
properties of the problem under study.  We now describe a graphical
method based on IDs for checking this property.

For each action node $A\in{\cal A}$ we identify two subsets,
$\pa{\co}{A}$ and $\pa{\ce}{A}$, of $\pa{\dag}{A}$, such that, when
$\sigma = \co$ [resp.\ $\ce$], the conditional distribution of $A$,
given its domain parents, can be chosen to depend only on
$\pa{\co}{A}$ [resp.\ $\pa{\ce}{A}$].

To ensure \condref{cont}, we suppose:
\begin{cond}
  \label{cond:contdag}
  $\pa{\ce}{A}\subseteq {\cal L}\cup{\cal A}$.
\end{cond}

In order to investigate \eqref{y} for a specific value of $i$, we now
construct, for $0 \leq i \leq N+1$, a new ID $\dag_i$ on ${\cal V}$,
as follows.  The only arrow out of $\sigma$ (again a founder node) is
now into $A_i$.  For $j < i$, the parent set of $A_j$ is
$\pa{\co}{A_j}$ with conditional distributions determined as under
$\co$; for $j > i$ it is $\pa{\ce}{A_j}$, with conditional
distributions determined as under $\ce$; finally, for $A_i$ it is
$(\parents{A_i}\,;\, \sigma)$, with conditional distributions exactly
as in $\dag$.  Any domain variable $V\in{\cal L}\cup{\cal U}$ has the
same parent set $\parents{V}$ (which will not include $\sigma$) and
conditional distributions as in $\dag$.  We shall use $\anc_i(\cdot)$
to denote a minimal ancestral set in $\dag_i$, with similar usages of
$\nd_i$, $\cip_i$, \etc

It is easy to see that the joint density of all the domain variables
in $ \dag_0 = \dag_\ce$ is $ p_0 = p_\ce$; in $\dag_{N+1} = \dag_\co$
it is $ p_{N+1} = p_\co$; while in $\dag_i$, given $\sigma = \co$ it
is $p_{i-1}$, and given $\sigma = \ce$ it is $p_{i}$.  Thus \eqref{y}
will certainly hold if
\begin{equation}
  \label{eq:graphsep}
  \inda Y {\sigma} {(\overline L_i, \overline A_i)} i
\end{equation}
holds.  We can easily check \eqref{graphsep} by inspection of the
graph $\dag_i$.  Note that $\dag_0$ is similar to the `manipulated'
DAG of \textcite{sgs:book}.

In summary we have shown the following:
\begin{theorem}
  \label{thm:extstabg}
  Under Conditions~\ref{cond:pos} and \ref{cond:contdag}, if the
  graphical separation property \eqref{graphsep} holds for each $i$,
  then we can compute the consequence of regime $\ce$ from data
  gathered under regime $\co$ by means of the $G$-recursion
  \eqref{standrecur}, starting with $f_N$ as in \eqref{fngen}, and
  ending with $f_0 = p(y\,;\,\ce)$.
\end{theorem}

A variant of this approach is described in \textcite{jmr:lns}, and
works as follows.  Let $\dag_i'$ be obtained from $\dag_i$ by omitting
the node $\sigma$, and deleting all arrows out of $A_i$.  Because
moralization links in $\dag_i$ involving $\sigma$ can only be to
predecessors of $A_i$, it is not difficult to see there exists a path
from $Y$ to $\sigma$ avoiding $(\overline L_i, \overline A_i)$ in
$\ma_{\dag_i}(Y, \overline L_i, \overline A_i)$ if and only if there
exists such a path from $Y$ to $\parents{A_i}$ in $\ma_{\dag_i'}(Y,
\overline L_i, \overline A_i)$.  And the latter condition can in turn
be seen to be equivalent to the existence, in that graph, of a path
from $Y$ to ${A_i}$ avoiding $(\overline L_i, \overline A_{i-1})$.
Thus $\inda Y {\sigma} {(\overline L_i, \overline A_i)} i$ if and only
if $\inda Y {A_i} {(\overline L_i, \overline A_{i-1})} {\dag_i'}$.
Hence we can prove \eqref{y} by demonstrating the latter property.

It is shown in \textcite{apd/vd:uai08} that, under certain further
conditions --- informally, that each intermediate variable has some
influence on the response under the interventional regime --- when the
graphical method described above succeeds we can deduce that the
problem in fact exhibits simple stability with respect to the observed
information base.

\subsection{Examples}
\label{sec:examples}

\subsubsection{Stability}
\label{sec:relation}
We first show that the conditions of \secref{simple} are a special
case of those of \secref{grec}, by verifying that the construction of
\secref{check} works for the case of simple stability, as represented
by \figref{simpdiag}.  In this case the $(U_i)$ are absent, and, for
each domain variable $V$, $\pa{\ce}{V} = \pa{\co}{V} = \pre{V}$.  Thus
$\dag_i$ consists of the complete directed graph on $(L_1, A_1,
\ldots, L_N, A_N, Y)$, together with an additional regime node
$\sigma$ and an arrow from $\sigma$ to $A_i$.  \figref{simpdiagi}
shows these graphs for the case $N=2$, and \figref{simpdiagidash} the
corresponding graphs $\dag_i'$.

\begin{figure}[ht]
  \begin{centering}
    \epsfxsize=1.8in \epsfbox{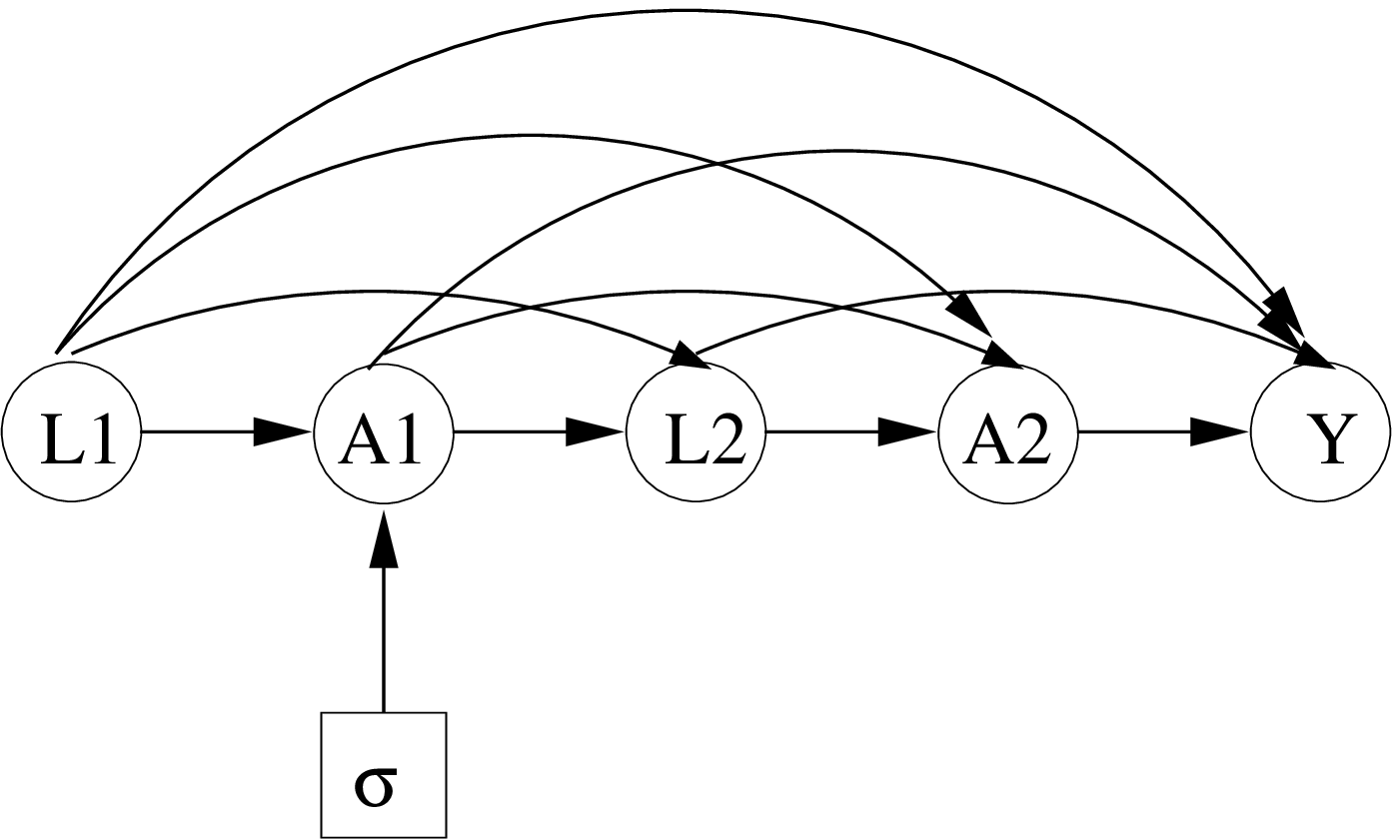} \quad\quad
    \epsfxsize=1.8in \epsfbox{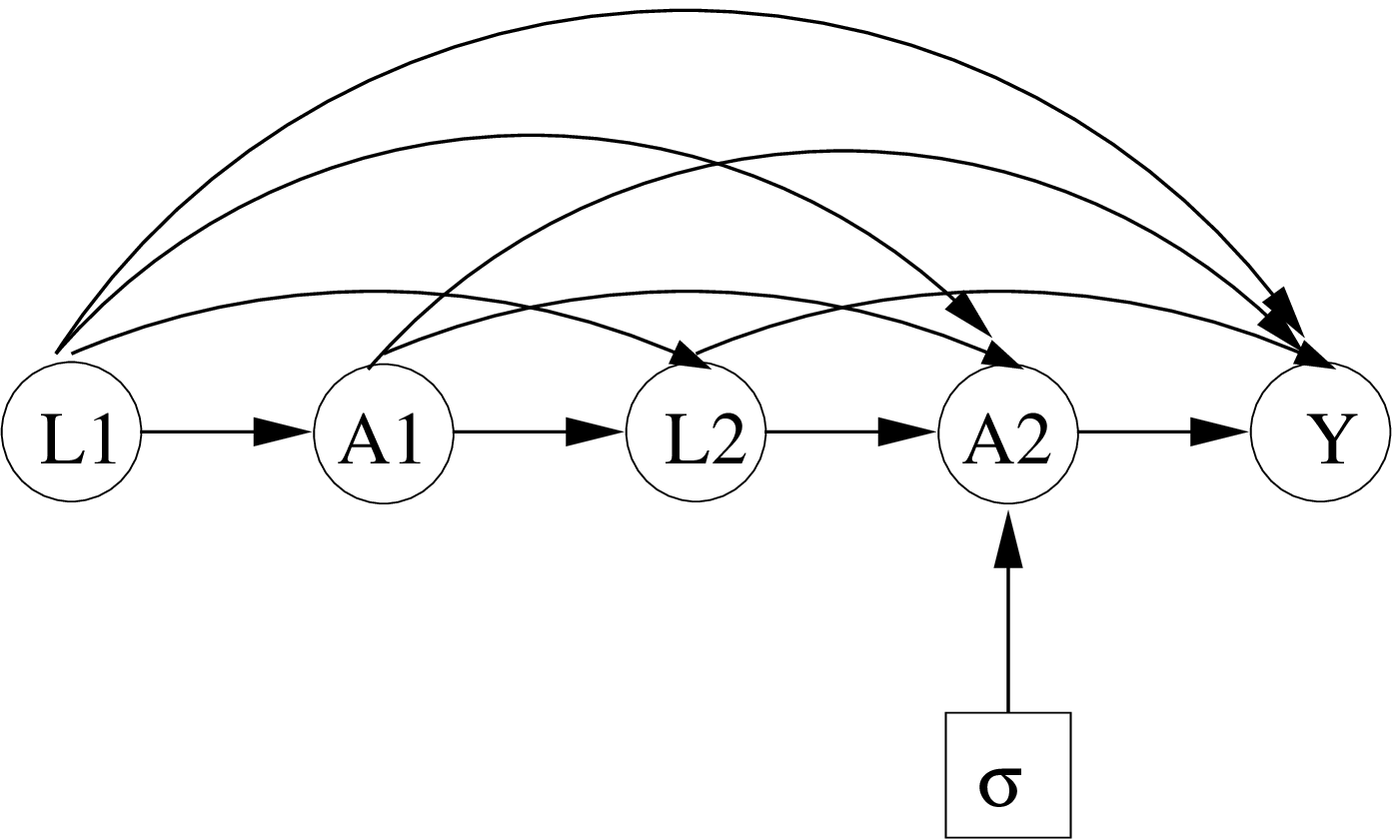}
    \caption{Influence diagrams $\dag_1$, $\dag_2$ for stability
      ($N=2$)}
    \label{fig:simpdiagi}
  \end{centering}
\end{figure}

\begin{figure}[ht]
  \begin{centering}
    \epsfxsize=1.8in \epsfbox{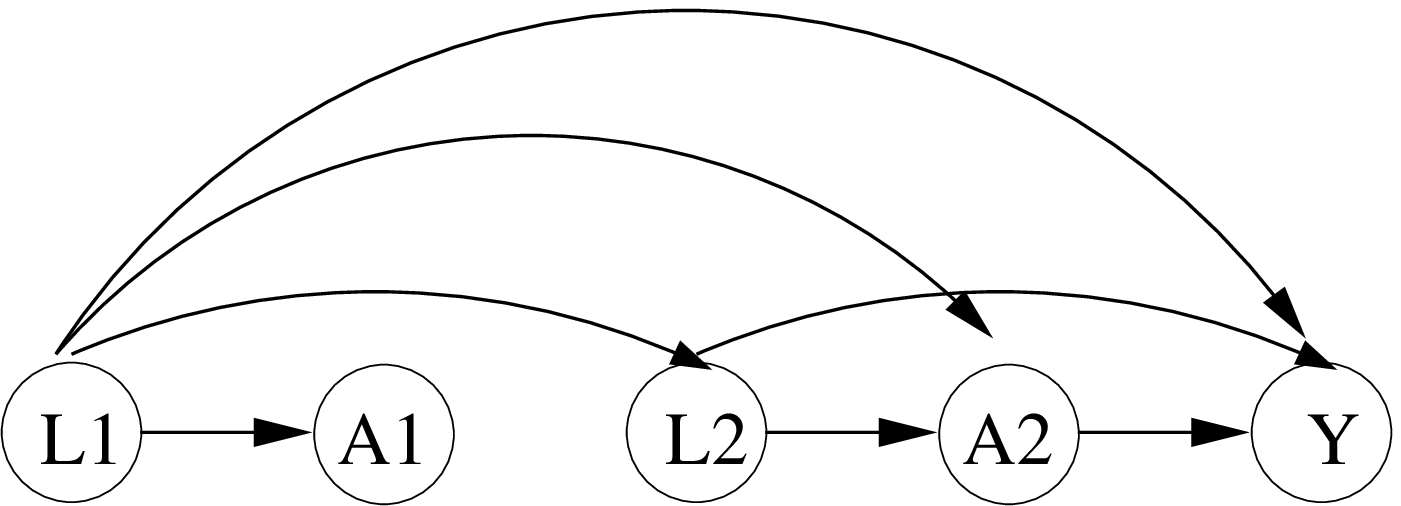} \quad\quad
    \epsfxsize=1.8in \epsfbox{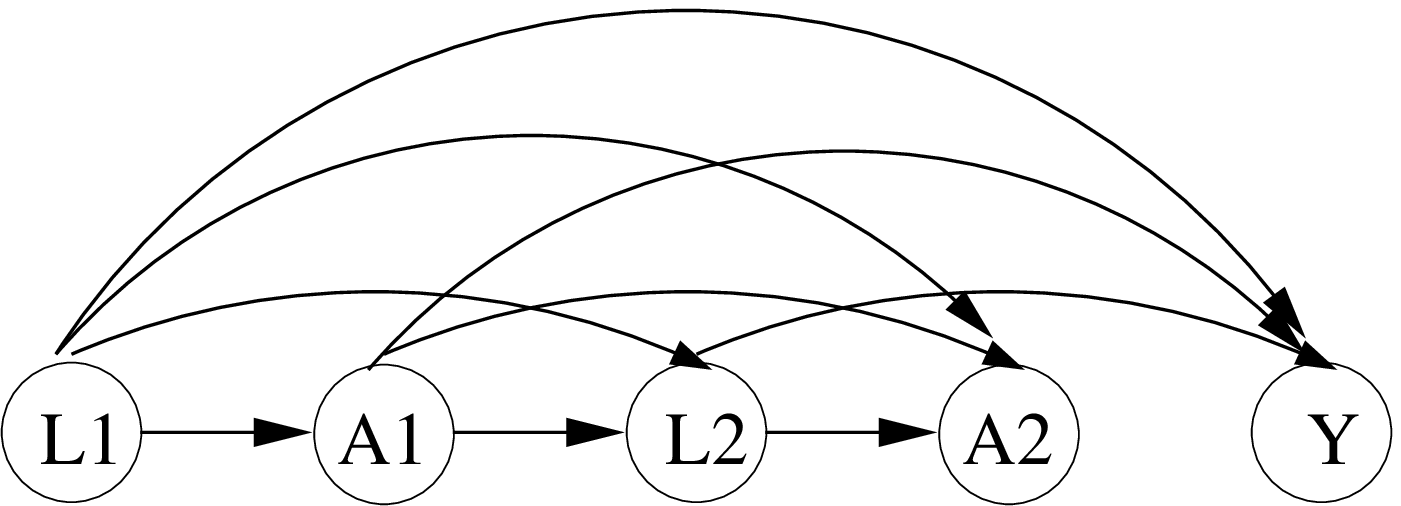}
    \caption{Influence diagrams $\dag_1'$, $\dag_2'$ for stability
      ($N=2$)}
    \label{fig:simpdiagidash}
  \end{centering}
\end{figure}

Since, after moralization of $\dag_i$, $\sigma$ has direct links only
into $(\overline L_i, \overline A_i)$, any path in this moral graph
joining $Y$ to $\sigma$ must intersect $(\overline L_i, \overline
A_i)$, whence we deduce \eqref{y}.  Equivalently, there is no path in
$\dag_i'$ from $Y$ to $A_i$ avoiding $(\overline L_i, \overline
A_{i-1})$.  Hence we have confirmed that, when stability holds, it is
possible to construct a sequence of joint densities $p_i$ satisfying
\eqref{N=o}--\eqref{y}.

\subsubsection{$G$-recursion without stability}
\label{sec:without}
More interesting is the possibility of applying the construction of
\secref{extstab} to justify $G$-recursion even in cases where simple
stability does not hold.  This is illustrated by the following
example, based on \textcite{pearl/robins:uai95} (and see
\textcite{jr:cma} and \textcite{jmr:lns} for description of medical
scenarios that are reasonably captured by this example).

\begin{ex}
  \label{ex:p+r}
  \figref{pearlrobins} shows a specific model incorporating extended
  stability for the information base $(U_1, A_1, U_2, L_2, A_2, Y)$
  (with $L_1 = \emptyset$).  Note that this does not embody simple
  stability, since moralization would create a direct link between
  $\sigma$ and $U_1$, and hence a path $L_2$---$U_1$---$\sigma$ that
  avoids $A_1$.  We thus can not deduce $\ind {L_2} \sigma {A_1}$, as
  would be required for simple stability.
  
  \begin{figure}[ht]
    \begin{centering}
      \epsfxsize=2.1in \epsfbox{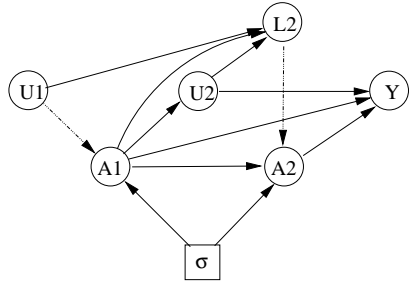}
      \caption{An ID displaying non-stability}
      \label{fig:pearlrobins}
    \end{centering}
  \end{figure}
  
  We use stippled arrows to represent independence under the control
  regime $\ce$.  Thus the stippled arrow from $U_1$ to $A_1$ in
  \figref{pearlrobins} represents the property
  \begin{equation}
    \label{eq:stipple1}
    \ind {A_1} {U_1} {\sigma = \ce},
  \end{equation}
  which is \eqref{controlpar} for $i=1$.  (The equivalent property for
  $i=2$ is already implied by the lack of any arrows from $U_1$ and
  $U_2$ to $A_2$).
  
  The stippled arrow from $L_2$ to $A_2$ embodies an additionally
  assumed property:
  \begin{equation}
    \label{eq:stipple2}
    \ind {A_2} {L_2} {(A_1\,;\, \sigma = \ce)}.
  \end{equation}
  That is, we are supposing that interventional assignment of $A_2$
  can only depend (deterministically or stochastically) on the value
  chosen for the previous treatment, $A_1$.  This is a restriction on
  the type of interventional strategy $\ce$ that we are considering.
  It will turn out that we can identify the causal effect of $\ce$
  from the observational data gathered under $\co$, using
  $G$-recursion, only for strategies $\ce$ of this special type.
  
  In this problem we thus have $\pa{\co}{A_1} = U_1$, $\pa{\ce}{A_1} =
  \emptyset$, $\pa{\co}{A_2} = (A_1, L_2)$, $\pa{\ce}{A_2} = A_1$.
  The constructed IDs $\dag_1$ and $\dag_2$ are shown in
  \figref{pearlrobinsi}, and the variant forms $\dag_1'$ and $\dag_2'$
  \cite[Figure~2]{pearl/robins:uai95} in \figref{pearlrobinsidash}.

  \begin{figure}[h]
    \begin{centering}
      \epsfxsize=1.8in \epsfbox{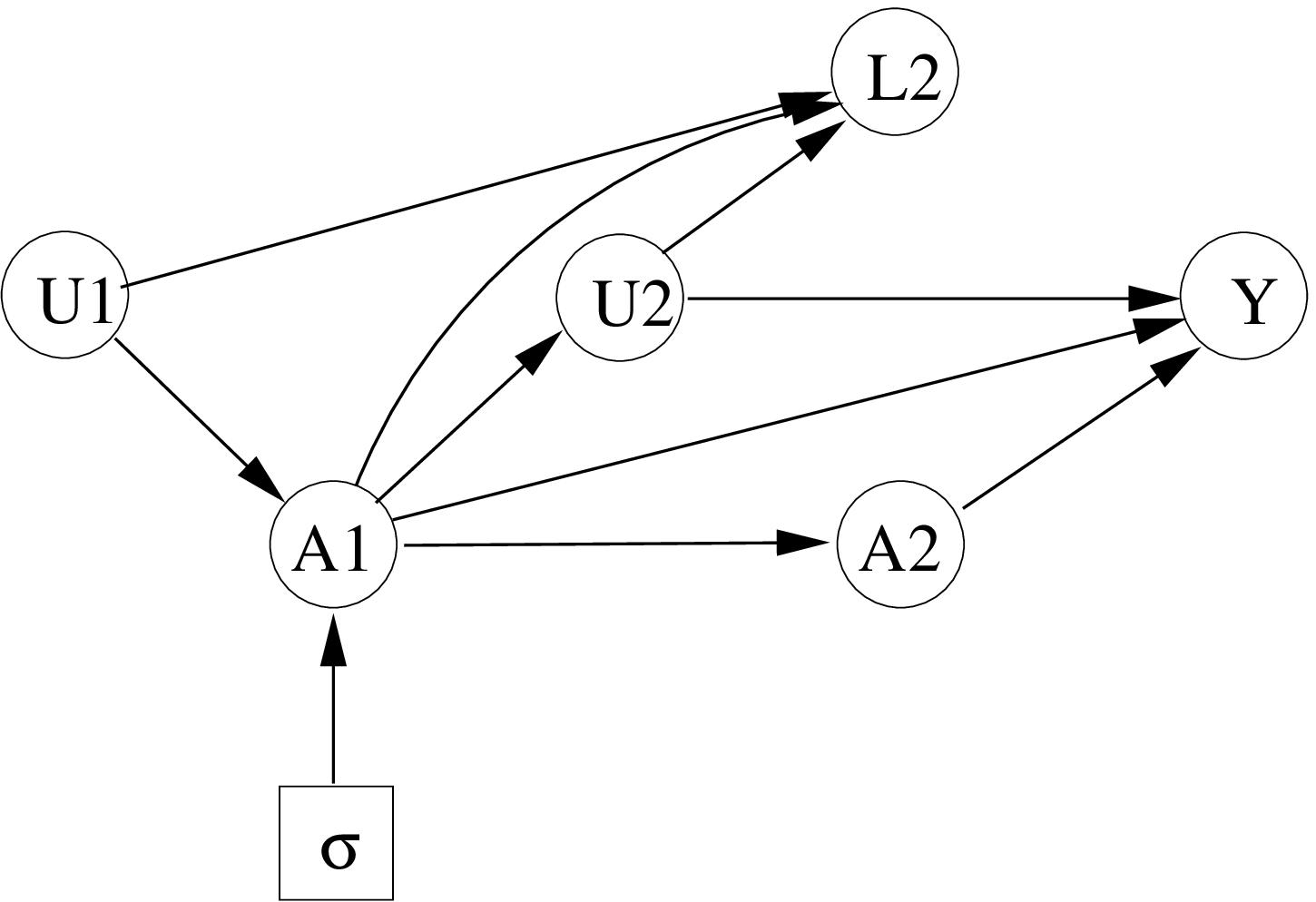} \quad\quad
      \epsfxsize=1.8in \epsfbox{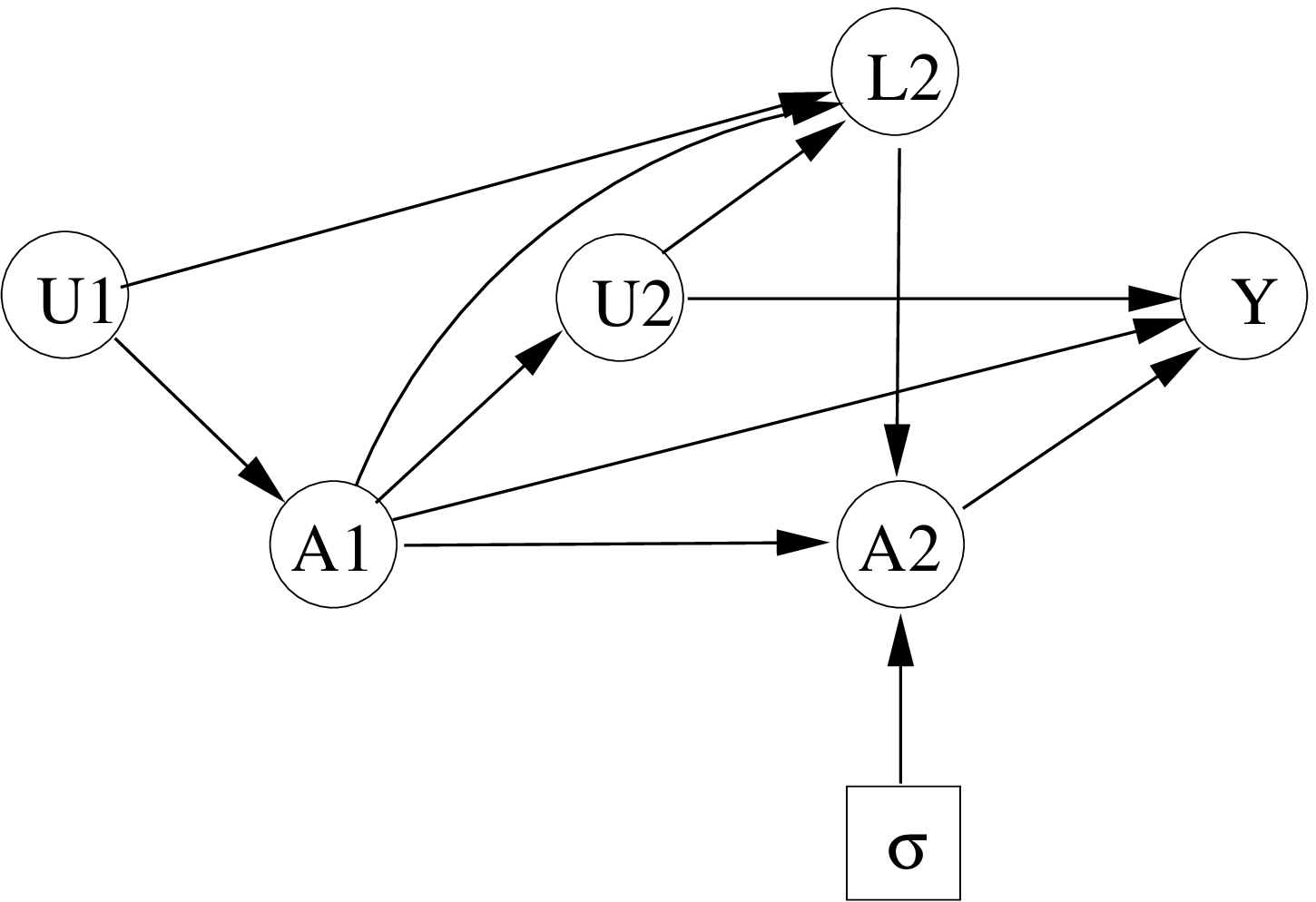}
      \caption{Influence diagrams $\dag_1$, $\dag_2$ for
        \figref{pearlrobins}}
      \label{fig:pearlrobinsi}
    \end{centering}
  \end{figure}
  
  \begin{figure}[h]
    \begin{centering}
      \epsfxsize=1.8in \epsfbox{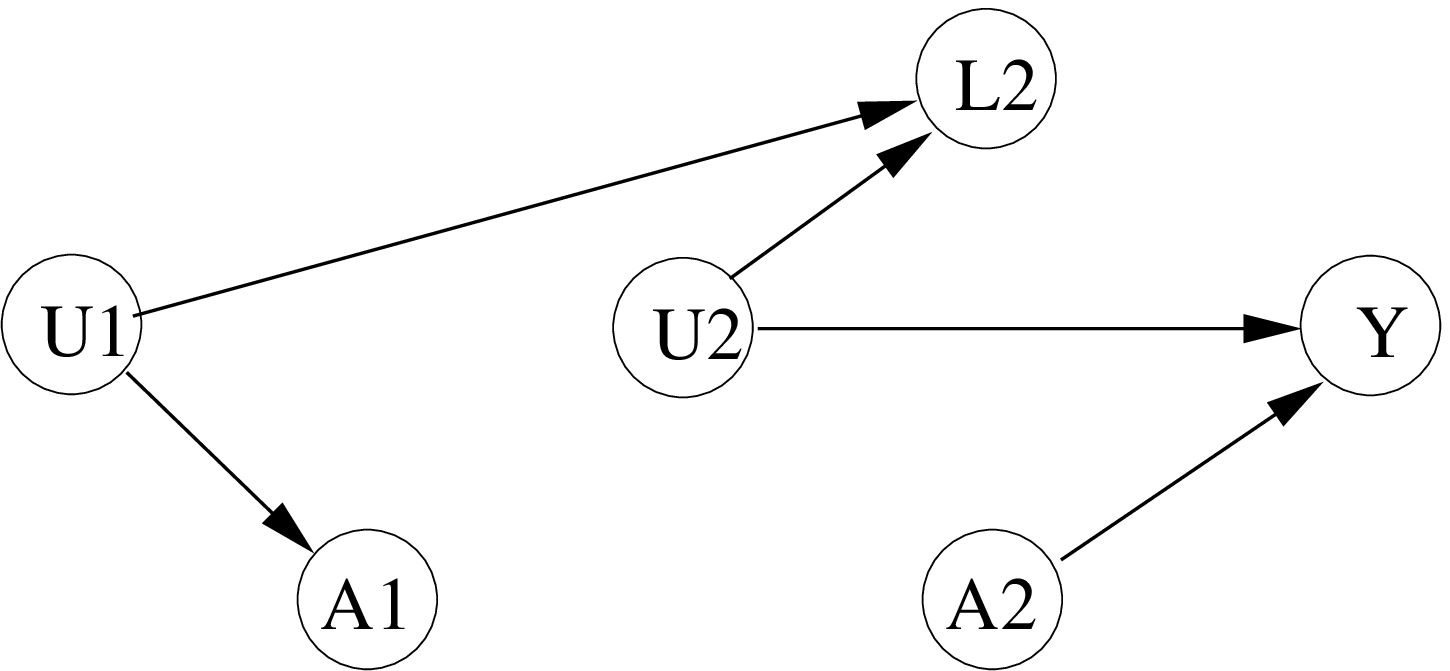} \quad\quad
      \epsfxsize=1.8in \epsfbox{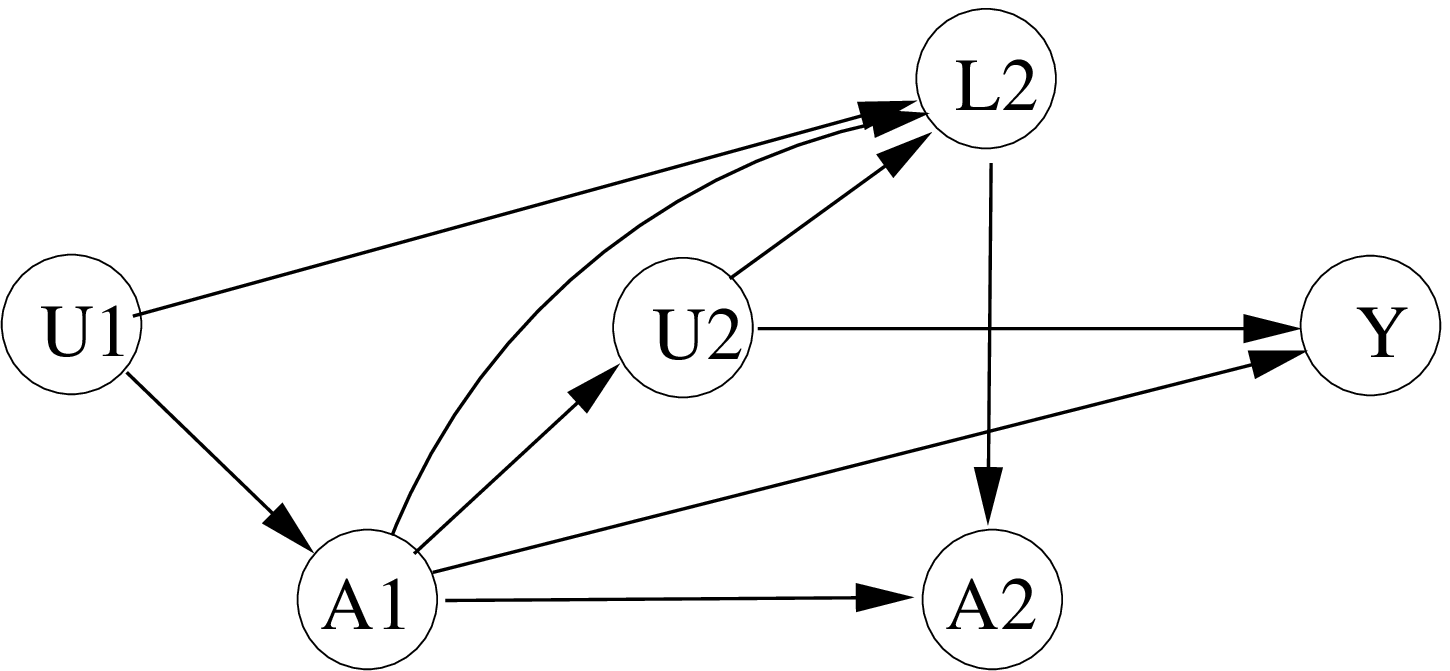}
      \caption{Influence diagrams $\dag_1'$, $\dag_2'$ for
        \figref{pearlrobins}}
      \label{fig:pearlrobinsidash}
    \end{centering}
  \end{figure}
  
  \begin{figure}[h]
    \begin{centering}
      \epsfxsize=2.1in \epsfbox{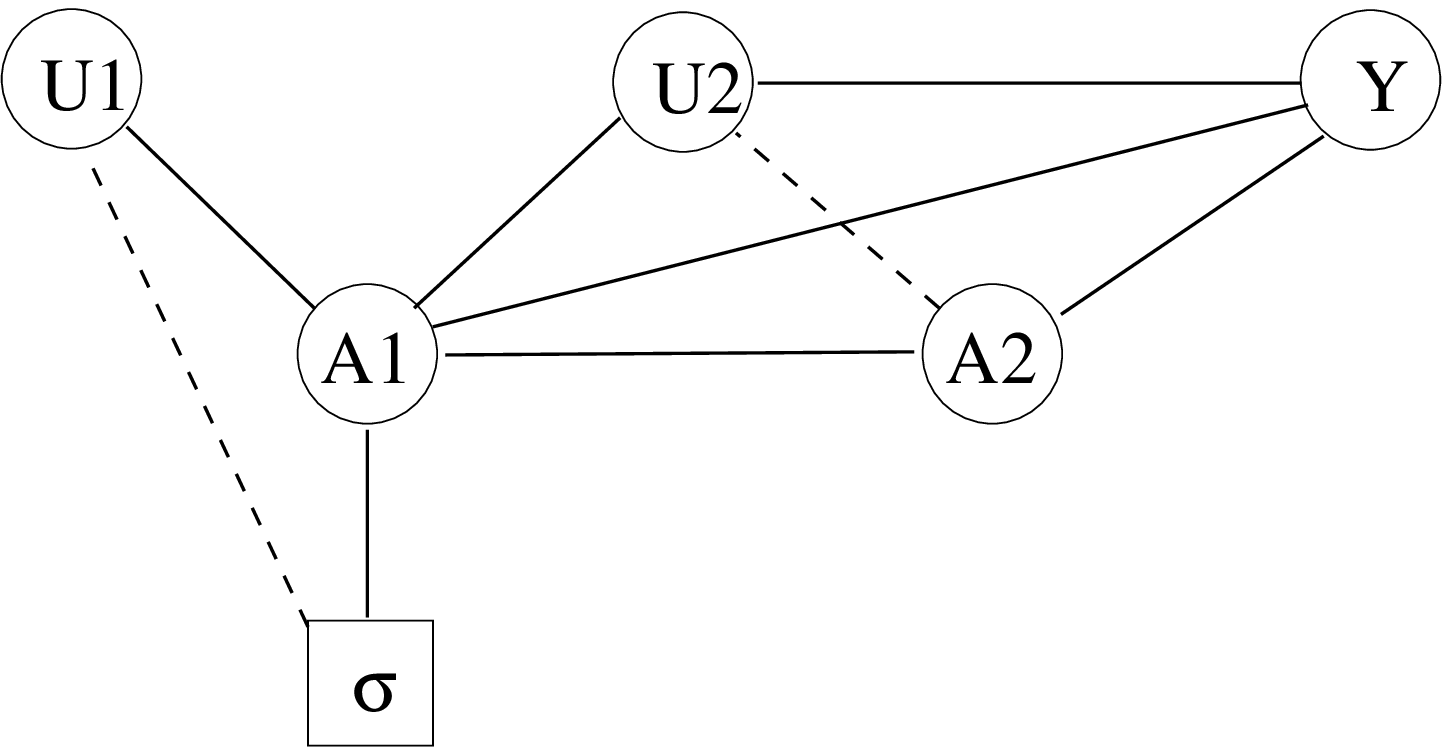} \quad\quad\quad
      \epsfxsize=2.1in \epsfbox{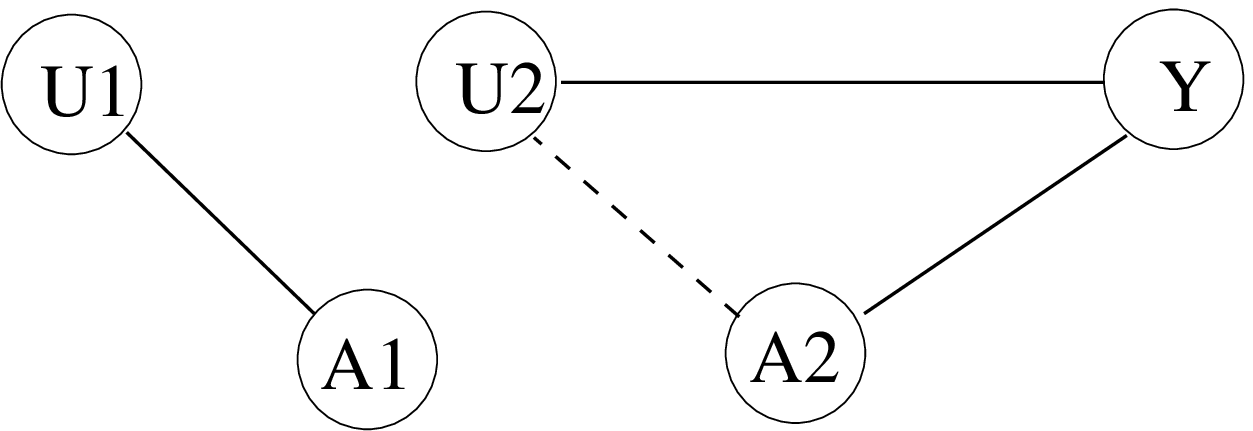}
      \caption{Relevant moral ancestral graphs, for $\dag_1$ and
        $\dag_1'$}
      \label{fig:pearlrobins4}
    \end{centering}
  \end{figure}
  
  We first examine $\dag_1$ to see if $\inda Y {\sigma} {A_1}
  {\dag_1}$.  The relevant moral ancestral graph (see
  \figref{pearlrobins4}) is easily seen to have the desired separation
  property: thus we have shown \eqref{y} for $i=1$.  Alternatively,
  from examination of the relevant moral ancestral graph based on
  $\dag_1'$ we readily see the desired property $Y \,\cip_{\dag_1'}
  A_1$.  (Note that this approach does {\em not\/} succeed if we allow
  $A_2$ to depend on $L_2$ under $\ce$, thus retaining an arrow from
  $L_2$ to $A_2$ and so making $L_2$ an ancestor of $Y$ in $\dag_1$:
  in the now larger relevant moral ancestral graph formed from
  $\dag_1$ we could then trace a path $Y$---$U_2$---$U_1$---$\sigma$
  from $Y$ to $\sigma$ avoiding $A_1$.)

  Finally, since in $\dag_2$ neither $U_1$ nor $U_2$ is a parent of
  $A_2$, even after moralization there will be no direct link from
  $\sigma$ to either $U_1$ or $U_2$: consequently any path from $Y$ to
  $\sigma$ will have to intersect $(A_1, L_2, A_2)$.  Equivalently, we
  see that in $\dag_2'$, after moralization (which adds a futher link
  between $U_1$ and $U_2$) every path from $Y$ to $A_2$ intersects
  $(A_1, L_2)$.  We deduce $\ind Y {\sigma} {(A_1, L_2, A_2)}$, \ie
  \eqref{y} for $i=2$.
  
  If we now assume Conditions~\ref{cond:pos} and \ref{cond:cont} then,
  all the required conditions being satisfied, we will have justified
  use of $G$-recursion to identify the consequences of an
  interventional regime $\ce$ of the specified form, from data
  collected under the observational regime $\co$.
\end{ex}
  
The graphical check illustrated above simplifies considerably in the
case of an unconditional interventional strategy $\ce$, where the
values of the action variables are determined in advance, as
considered by \textcite{pearl/robins:uai95}. In this case
$\pa{\ce}{A_i} = \emptyset$ for all $i$, and $\dag_i$ is obtained from
$\dag$ by deleting all arrows into every $A_j$ with $j > i$.  Then
$\dag_i'$ is obtained by further deleting $\sigma$ and all arrows out
of $A_i$.  However, if our aim is to compare strategies, and ideally
find an optimal one, it is necessary also to consider dynamic
strategies.

\section{Constructing an admissible sequence}
\label{sec:vdproof}

In order to apply the graphical check of \secref{check} we need to
have the variables already completely ordered.  More generally, we
could ask whether there exists an ordering $(A_1, \ldots, A_N)$ of
${\cal A}$, and $(L_1, \ldots, L_N)$ of disjoint subsets of ${\cal
  L}$, such that we can apply the construction of \secref{check} to
show \eqref{graphsep}.  Somewhat more restricted, we might suppose an
ordering $(A_1, \ldots, A_N)$ already given, and look for a sequence
$(L_1, \ldots, L_N)$ to satisfy \eqref{graphsep}.  Such a sequence
will be termed {\em admissible\/}.  In this section we assume that a
graphical representation of the problem in form of an ID is given, and
we note that by definition an admissible sequence has to satisfy
$\overline L_i \subseteq \nd (A_i, \ldots,A_N)$. Below, we give
conditions under which we can determine whether such an admissible
sequence exists, and construct one if it does.  We shall need some
general properties of directed-graph separation from
\appref{problemmas}.

We impose the following conditions:
\begin{cond}
  \label{cond:pain}
  For all $i$,
  \begin{displaymath}
    \pa{\ce}{A_i} \subseteq \pa{\co}{A_i}.
  \end{displaymath}
\end{cond}
This can always be ensured by redefining, if necessary,
$\pa{\co}{A_i}$ as $\pa{\co}{A_i}\cup\pa{\ce}{A_i}$, with any added
parents having no effect on the conditional probabilities for $A_i$
under $\co$.

\begin{cond}
  \label{cond:actan}
  Each action variable $A\in{\cal A}$ is an ancestor of $Y$ in
  $\dag_\ce$.
\end{cond}
In typical contexts \condref{actan} will hold, since we would not
normally contemplate an intervention that has no effect on the
response.  Clearly when Conditions~\ref{cond:pain} and
\ref{cond:actan} both hold every $A\in{\cal A}$ is also an ancestor of
$Y$ in $\dag_\co = \dag$.

Define, for $i = 1, \ldots, N$:
\begin{equation}
  \label{eq:m}
  M_i: = {\cal L} \cap \nd_\ce(A_i, A_{i+1}, \ldots, A_N) \cap\anc_i(Y).
\end{equation}

We note that $M_{i-1} \subseteq M_i$.  This follows from
$\anc_{i-1}(Y)\subseteq \anc_i(Y)$ which in turn holds because, by
\condref{pain}, the edge set of $\dag_{i-1}$ is a subset of that of
$\dag_i$.

Now let
\begin{equation}
  \label{eq:L}
  L^*_i: = M_i\setminus M_{i-1},
\end{equation}
so that $M_i = \bar{L}^*_i$.  For the information sequence $(L^*_i)$,
the total information taken into account up to time $i$, $M_i$,
consists of just those variables in ${\cal L}$ that are ancestors of
$Y$ in $\dag_i$, but are not descendants of $A_i$ or any later
actions.

The sequence $(L^*_1, \ldots, L^*_N)$ will be admissible if, for $i =
1, \ldots, N$,
\begin{equation}
  \label{eq:lstar}
  \inda Y {\sigma} {(M_i, \overline{A}_i)} i.
\end{equation}
  
Taking into account \condref{actan} and \eqref{m}, \eqref{lstar}
requires that $M_i \cup \overline{A}_i$ separate $Y$ from $\sigma$ in
the undirected graph $\graf_i$ obtained by moralizing the ancestral
set of $Y$ in $\dag_i$.  It is thus straightforward to check whether
or not it holds.  When it does we shall call $i$ {\em admissible\/}.

The following result can be regarded as simultaneously simplifying,
generalizing, and rendering more operational that of
\textcite{pearl/robins:uai95}.  In particular, it supplies an explicit
construction, while allowing for conditional interventions.

\begin{theorem}
  \label{thm:vd}
  Under Conditions~\ref{cond:pain} and \ref{cond:actan}, if any
  admissible sequence exists then $(L_1^*, \ldots, L_N^*)$ is
  admissible.
\end{theorem}
That is: There exists an admissible sequence if and only if every $i$
is admissible.  In this case $(L^*_1, \ldots, L^*_N)$ is an admissible
sequence.\\

\begin{proof}
  Suppose that there exists some admissible sequence $(L_1, \ldots,
  L_N)$.  Then, for each $i$,
  \begin{equation}
    \label{eq:aa}
    \inda Y {\sigma} {\overline L_i\cup \overline A_i} i.
  \end{equation}
  By \lemref{pr2}, this graph-theoretical separation continues to hold
  if we intersect the conditioning set with $(\anc_i(Y), \sigma)$.
  Since, by \condref{actan}, $\overline A_i\subseteq\anc_i(Y)$, we
  obtain
  \begin{equation}
    \label{eq:bb}
    \inda Y {\sigma} {\left(\overline
        L_i\cap\anc_i(Y)\right)\cup\overline A_i} i. 
  \end{equation}
  But, $\overline L_i\subseteq {\cal L}\cap\nd_e(A_i, \ldots,A_N)$;
  and thus $\overline L_i \cap\anc_i(Y) \subseteq M_i$.  Hence, by
  \lemref{pr2}, \eqref{lstar} holds, and the result follows.
\end{proof}

\begin{ex} (We are indebted to Susan Murphy for this example.)  In the
  problem represented in \figref{susan}, it may be checked that the
  `obvious' choice $L_1 = \{X\}, L_2 = \{Z\}$ is not an admissible
  sequence.  Using the method above yields $L^*_1 = \{X, Z\}, L^*_2 =
  \emptyset$, which is admissible (indeed, yields simple stability, as
  may either be checked directly, or deduced from Theorem 2 in
  \textcite{apd/vd:uai08}).
  \begin{figure}[ht]
    \begin{centering}
      \epsfxsize=2.8in \epsfbox{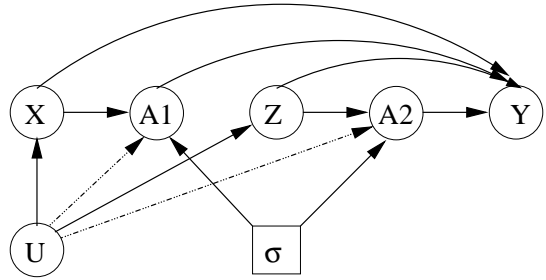}
      \caption{Finding an admissible sequence}
      \label{fig:susan}
    \end{centering}
  \end{figure}
\end{ex}

\subsection{Finding a better sequence}
\label{sec:min}
While the above procedure will always construct an admissible sequence
$(L_1, \ldots, L_N)$ when one exists, that might not be the best
possible.  Thus in \figref{nonmin}, with ${\cal L} = \{X, Z\}$, we
find $L_1^* = \{Z\}, L_2^* = \{X\}$.  These satisfy \eqref{lstar}, so
that the sequence $\{L_1^*, L_2^*\}$ is admissible.  However a smaller
admissible sequence is given by $L_1 = \emptyset, L_2 = \{X\}$.
\begin{figure}[h]
  \begin{center}
    \resizebox{1.8in}{!}{\includegraphics{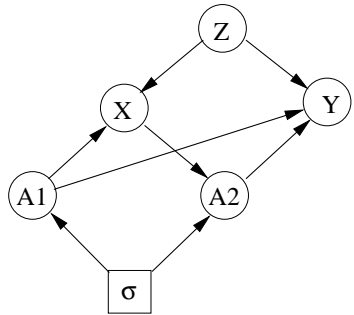}}
    \caption{A choice of admissible sequences}
    \label{fig:nonmin}
  \end{center}
\end{figure}

If we had initially regarded $Z$ as unobservable, so taking ${\cal L}
= \{X\}$, we would have found this smaller sequence.  However in
general we would need hindsight or good fortune to start off with such
a minimal specification of ${\cal L}$.

Even without redefining ${\cal L}$, however, we can often improve on
the sequence given by \eqref{L}.  At each stage $i$ we first check
\eqref{lstar}.  If this fails we abort the process.  Otherwise,
sequentially choose $L_i$ to be any subset of $M_i$, disjoint from
$\overline L_{i-1}$, such that \eqref{aa} holds.  (Since, by
\eqref{lstar}, \eqref{aa} holds for the choice $L_i = M_i\setminus
\overline L_{i-1}$, such a set must exist.)  Then (if the process is
never aborted) we shall have constructed an admissible sequence
$(L_i)$, improving on $(L^*_i)$ in the sense that $\overline L_i
\subseteq \overline L^*_i$.

Ideally we would want the set $L_i$ to be small.  When each $L_i$ is
minimal, in the sense that no proper subset of $L_i$ satisfies
\eqref{aa}, we obtain a generalization of the method of
\textcite{pearl/robins:uai95} for constructing a minimal admissible
sequence.  However in large problems the search for such a minimal
$L_i$ can be computationally non-trivial, and we may have to be
satisfied with some other choices for the $(L_i)$.  Minimality is in
any case not a requirement for admissibility.

\subsection{Admissible orderings of ${\cal A}$}
\label{sec:adma}
In general there will be several orderings of ${\cal A}$ possible.  It
can then happen that an admissible sequence $(L_1, \ldots, L_N)$
exists for one ordering of ${\cal A}$ (which we may then likewise call
{\em admissible\/}), but not for another.

\begin{ex}
  \label{ex:vdex}
  In the ID of \figref{vdpic}, ${\cal U} = \{U\}$, ${\cal L} = \{L\}$,
  ${\cal A} = \{A, B\}$.  Note that $\indo A B$ under either regime.
  Both $A_1 = A, A_2 = B$ and $A_1 = B, A_2 = A$ are possible
  orderings of ${\cal A}$.  For the former choice we find $M_1 =
  \emptyset$; then \eqref{lstar} for $i=1$ becomes $\inda Y {\sigma}
  {A} {\dag_A}$, where $\dag_A$ is $\dag$ with the arrow from $\sigma$
  to $B$ removed.  Since this is easily seen to fail (moralization
  creates a link between $U$ and $\sigma$), \thmref{vd} implies that
  there can be no admissible sequence to support $G$-recursion.
  However if we take $A_1 = B, A_2 = A$, we obtain $M_1 = \emptyset$,
  $M_2 = \{L\}$, and \eqref{lstar} becomes $\inda Y {\sigma} B
  {\dag_B}$ for $i=1$, $\inda Y {\sigma} {(B, L, A)} {\dag_A}$ for
  $i=2$, where $\dag_B$ is $\dag$ with the arrows into $A$ from both
  $\sigma$ and $U$ removed.  Both of these properties are easily
  confirmed to hold.  We can thus (under suitable positivity
  conditions) apply $G$-recursion with respect to the admissible
  ordering $(B, L, A)$.
\end{ex}

As yet we do not have a method that will automatically identify an
admissible ordering of ${\cal A}$ when one exists.

\begin{figure}[h]
  \begin{center}
    \resizebox{1.8in}{!}{\includegraphics{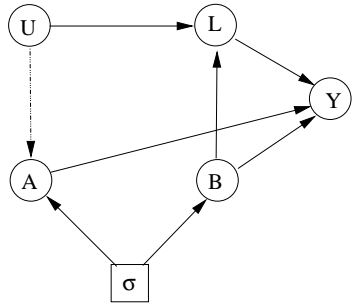}}
    \caption{Unordered actions}
    \label{fig:vdpic}
  \end{center}
\end{figure}

\section{Potential response models}
\label{sec:potresp}
  
In this section, we examine the relationship between the {\em
  potential response (PR)\/} approach to dynamic treatments and our
own decision-theoretic one.

The PR approach typically confines attention to non-randomized, though
possibly dynamic, strategies.  Such a strategy is defined by a
function $g$ on the set of all `partial $L$-histories' of the form
$(\overline l_i)$ ($1 \leq i \leq N)$, such that, for each $i$,
$g(\overline l_i)$ is one of the available options for $A_i$.  We
shall write $\overline g(\overline l_i)$ for the sequence $(g(l_1),
g(l_1,l_2), \ldots, g(\overline l_i))$.  Under this strategy, if at
time $i$ we have observed $\overline L_i = \overline l_i$, the next
action will be $A_i = g(\overline l_i)$.

We henceforth confine attention to a pair of regimes ${\cal S} =
\{\co, \ce\}$, where $\co$ is observational, while $\ce$ is a
non-randomized strategy, determined by a given function $g$ as
described above.

\subsection{Potential responses and stability}
\label{sec:stablerobins}
We first interpret and analyse the model introduced by
\textcite{jr:mm} (see also \textcite{jmr:lns}, Section 3.3;
\textcite{jmr:sbs00}; \textcite{susan}).

We need to introduce, for each regime $s\in \{\co, \ce\}$, a
collection of `potential variables' $\Pi_s := (L_{s,1}, A_{s,1},
\ldots, L_{s,N}, A_{s,N}, L_{s,N+1} \equiv Y_s )$.  It is supposed
that, when regime $s$ is operating, the actual observable variables in
the problem, $(L_1, A_1, \ldots, L_N, A_N, Y)$, will be those in
$\Pi_s$.

Note that, by the definition of $\ce$, we have the functional
constraint
\begin{equation}
  \label{eq:ds}
  A_{\ce, i} = g(\overline L_{\ce , i})\quad(i=1, \ldots, N).
\end{equation}

All the potential variables, across both regimes, are regarded as
having simultaneous existence, their values being unaffected by which
regime is actually followed.%
\footnote{%
  Note, as a matter of logic, that if we follow $\ce$ we shall not be
  able to observe \eg\ $Y_{\co}$ (though see the note after
  \condref{consistency} below).  This is a version of the so-called
  `fundamental problem of causal inference' \cite{pwh:jasa} which has
  been critically discussed by \textcite{apd:cinfer}.%
} %
The effect of following regime $s$ is thus to uncover the values of
some of these, \viz those in $\Pi_s$, while hiding others.

This collection of all potential observables across both regimes is
further considered to have a joint distribution (respecting the
logical constraints \eqref{ds}), whose density we denote by
$p(\cdot)$.  This distribution is supposed unaffected by which regime
is in operation: all this can do is change the relationship between
potential and actual variables.

Since, under $\ce$, $Y \equiv Y_\ce$, the consequence of the
interventional strategy $\ce$ is simply the marginal distribution of
$Y_\ce$.  Our aim is to identify this distribution from observations
made under regime $\co$.

It can be shown directly that this can be effected by means of the
$G$-recursion formula under the following conditions:
\begin{cond}[Positivity]
  \label{cond:ac1a}
  Whenever $p(\overline L_{\co,N} = \overline l_N) > 0$,
  \begin{displaymath}
    p(\overline A_{\co,N} = \overline g(\overline l_N) \mid \overline L_{\co,N} = \overline l_N) > 0.
  \end{displaymath}
\end{cond}
That is, in the observational regime, for any set of values $\overline
l_N$ of the variables $\overline L_N$ that can arise with positive
probability, there is a positive probability that the actions taken
will be those specified by $\ce$.

\begin{cond}[Consistency]
  \label{cond:consistency}
  If $\overline A_{\co,i} = \overline g(\overline L_{\co,i})$, then
  $L_{\co, i+1} = L_{\ce, i+1}$ $(i=0, \ldots, N)$.
\end{cond}
(Note that for $i=0$ the antecedent of this condition is vacuously
satisfied, while for $i=N$ its conclusion is $Y_\co = Y_\ce$.)

That is, if, in the observational regime, we happen to obtain a
partial history $(\overline l_i, \overline a_i)$ that could also be
obtained under the operation of $\ce$, then we will next observe the
identical variable $L_{\ce, i+1}$ that would have been observed if we
had been operating $\ce$.  (This condition of course imposes further
logical constraints on the joint distribution $p$).

\begin{cond}[Sequential ignorability]
  \label{cond:ignore}
  Whenever \begin{math} p(\overline L_{\co,i} = \overline l_i) > 0,
  \end{math}
  \begin{displaymath}
    \ind {A_{\co,i}} {\overline L_\ce^{i+1}} {(\overline L_{\co,i} =
      \overline l_{i}, \overline A_{\co,i-1} = \overline g(\overline
      l_{i-1}))} 
    \quad (i= 1, \ldots, N),
  \end{displaymath}
\end{cond}
(where $\overline L_\ce^j: = (L_{\ce,j}, L_{\ce,j+1}, \ldots,
L_{\ce,N}, Y_\ce)$).

That is, in the observational regime, given any partial history
consistent with the operation of $\ce$, the next action is independent
of all the future potential observables associated with $\ce$.%
\footnote{%
  This is sometimes expressed in a stronger form that drops the
  restriction to future variables, so replacing $L_\ce^{i+1}$ by
  $(\overline L_{\ce,N}, Y_\ce)$ \cite{jmr:sbs00}.  }%

\subsubsection{Connexions}
\label{sec:conn}
We now consider the relationship between the above approach and that
of \secref{simple}, which founds $G$-recursion on the stability
property \eqref{expci2}.  We will show that
Conditions~\ref{cond:ac1a}, \ref{cond:consistency} and
\ref{cond:ignore} imply our conditions in \secref{simple}.  Our
reasoning is, in spirit, very similar to Theorem 3.1 of
\textcite{jmr:lns} (see also \textcite{jr:mm}, Theorem 4.1).

\begin{lem}
  \label{lem:get}
  If Conditions~\ref{cond:consistency} and \ref{cond:ignore} hold,
  then for any sequence $\overline l_{N+1} = (l_1, \ldots, l_{N}, y)$
  such that $p(\overline L_{\ce, N} = \overline l_N) > 0$,
  \begin{equation}
    \label{eq:extend}
    p(\overline L_{\ce}^{i+1} = \overline l^{i+1} \mid \overline L_{\ce, i} = \overline
    l_i, \overline A_{\co,i} = \overline g(\overline l_i)) =
    p(\overline L_{\ce}^{i+1} = \overline l^{i+1} \mid \overline L_{\ce, i} =
    \overline l_i)
  \end{equation}
  for $i = 0, \ldots, N$.
\end{lem}

\begin{proof}
  First note that, from \condref{consistency}, when $\overline
  A_{\co,i} = \overline g(\overline l_i)$, $\overline L_{\co,i+1} =
  \overline l_{i+1}$ is equivalent to $\overline L_{\ce,i+1} =
  \overline l_{i+1}$.  So from \condref{ignore}
  \begin{equation}
    \label{eq:equiv}
    \ind {A_{\co,i+1}} {\overline L_\ce^{i+2}} {(\overline L_{\ce,i+1} = \overline l_{i+1}, \overline
      A_{\co,i} = \overline g(\overline l_i))}.
  \end{equation}
  
  We now show \eqref{extend} by induction on $i$.
  
  It holds trivially for $i=0$.  Suppose then it holds for $i$.
  Conditioning both sides on $L_{\ce,i+1} = l_{i+1}$ then yields
  \begin{displaymath}
    p(\overline L_\ce^{i+2} = \overline l^{i+2} \mid \overline L_{\ce, i+1} =  \overline
    l_{i+1}, \overline A_{\co,i} =  \overline g(\overline l_{i})) =  p(\overline L_\ce^{i+2} =
    \overline l^{i+2} \mid \overline L_{\ce, i+1} =  \overline l_{i+1}). 
  \end{displaymath}
  
  But from \eqref{equiv} we have
  \begin{displaymath}
    p(\overline L_\ce^{i+2} = \overline l^{i+2} \mid \overline L_{\ce, i+1} 
    =  \overline
    l_{i+1}, \overline A_{\co,i+1} =  \overline g(\overline l_{i+1}))
  \end{displaymath}  
  \begin{displaymath}
    = 
    p(\overline L_\ce^{i+2} = \overline l^{i+2} \mid \overline L_{\ce, i+1} =  \overline
    l_{i+1}, \overline A_{\co,i} =  \overline g(\overline l_{i})).
  \end{displaymath}
  
  Hence \eqref{extend} holds with $i$ replaced by $i+1$ and the
  induction proceeds.
\end{proof}

\begin{thm}
  If Conditions~\ref{cond:consistency} and \ref{cond:ignore} hold,
  then so does the stability condition \eqref{expci2}.
\end{thm}

\begin{proof}
  Because of \eqref{ds}, and the restriction immediately below the
  density interpretation \eqref{as} of \eqref{expci2}, it is enough to
  show that $p(L_{\ce,i+1} = l_{i+1} \mid \overline L_{\ce,i} =
  \overline l_{i}, \overline A_{\ce,i} = \overline g(\overline l_{i}))
  = p(L_{\co,i+1} = l_{i+1} \mid \overline L_{\co,i} = \overline
  l_{i}, \overline A_{\co,i} = \overline g(\overline l_{i}))$.  But,
  again by \eqref{ds}, $p(L_{\ce,i+1} = l_{i+1} \mid \overline
  L_{\ce,i} = \overline l_{i}, \overline A_{\ce,i} = \overline
  g(\overline l_{i})) = p(L_{\ce,i+1} = l_{i+1} \mid \overline
  L_{\ce,i} = \overline l_{i})$.  By \lemref{get}, this is the same as
  $p(L_{\ce, {i+1}} = l_{i+1} \mid \overline L_{\ce, i} = \overline
  l_{i}, \overline A_{\co,i} = \overline g(\overline l_{i}))$, and by
  \condref{consistency} this is in turn the same as $p(L_{\co,i+1} =
  l_{i+1} \mid \overline L_{\ce,i} = \overline l_{i}, \overline
  A_{\co, i} = \overline g(\overline l_{i}))$.
\end{proof}

Finally, in the light of \eqref{ds}, it is easy to see that
\condref{ac1a} implies  positivity as given by \defref{positive}.\\

In summary, whenever the conditions usually used to justify
$G$-recursion in the potential response framework hold, so will our
own (as in \secref{simple}).  But our conditions are more general in
that they do not require the existence of, let alone any probabilistic
relationships between, potential responses under different regimes;
and can, moreover, just as easily handle randomized interventional
strategies, which are more problematic for the PR approach.

\subsection{Potential responses without stability}
\label{sec:judith}
A more general approach
\cite{jr:cma,jmr:hsrm,robins/hernan/brumback:marg_struct,gill/robins:01,lok:04}
within the potential response framework replaces
Conditions~\ref{cond:consistency} and \ref{cond:ignore} with the
following variants:

\begin{cond}
  \label{cond:consistency01}
  If $\overline A_{\co,N} = \overline g(\overline L_{\co,N})$, then
  $Y_\co = Y_\ce$.
\end{cond}
  
That is, if in the observational regime we happen to observe a
complete history that could have arisen under the operation of $\ce$,
then the response will be identical to what we would have observed had
we been operating $\ce$.

\begin{cond}
  \label{cond:ignore1}
  \begin{displaymath}
    \ind {A_{\co,i}} {Y_\ce} {(\overline L_{\co,i} = \overline l_i, \overline
      A_{\co,i-1} = \overline g(\overline l_{i-1}))}\quad (i= 1, \ldots, N).
  \end{displaymath}
\end{cond}
That is, if, in the observational strategy, we happen to observe a
partial history that could have arisen under the operation of $\ce$,
then the next action is independent of the potential response under
$\ce$.

\condref{consistency01} implies, and can in fact be replaced by:
\begin{cond}
  \label{cond:consistency1}
  Given $(\overline L_{\co,N} = \overline l_N, \overline A_{\co,N}
  =\overline g(\overline l_N))$, $Y_\co$ and $Y_\ce$ have the same
  conditional distribution.
\end{cond}

The deterministic strategy $\ce$ is termed {\em evaluable\/} if, for
each $i$:
\begin{cond}
  \label{cond:eval}
  \begin{displaymath}
    p\left(\overline L_{\co,i} = \overline l_i, \overline A_{\co,i-1} =
      g(\overline l_{i-1})\right) > 0 \Rightarrow p\left(\overline L_{\co,i} = \overline l_i, \overline A_{\co,i} =
      g(\overline l_{i})\right) > 0.
  \end{displaymath}
\end{cond}

Note that Conditions~\ref{cond:consistency01}--\ref{cond:eval} make no
mention of potential intermediate variables $(\overline L_{\ce,N},
\overline A_{\ce,N})$ under $\ce$ --- though they do involve both
versions, $Y_\co$ and $Y_\ce$, of the response.  The relevant
variables in the problem can thus be taken as $(\overline L_{\co,N},
\overline A_{\co,N}, Y_\co, Y_\ce)$, having a
joint distribution $p$ say.\\

Conditions~\ref{cond:ignore1} and \ref{cond:consistency1} are weaker
than those of \secref{stablerobins} as none of the variables under
strategy $\ce$ other than $Y_\ce$ are involved.  Note that, for
example, it is not required that, when an observational partial
history could have arisen under $\ce$, that is the history that would
have so arisen; but even so, constraints on $Y_\ce$ are then imposed.

\subsubsection{Connexions}
\label{sec:conn1}
It is straightforward to show directly that, when
Conditions~\ref{cond:ignore1}, \ref{cond:consistency1} and
\ref{cond:eval} hold, the marginal distribution of $Y_e$, or the
interventional consequence $\E\{k(Y_e)\}$, can be identified by the
$G$-recursion \eqref{determstandrecur}.  We now show how this approach
can be related to our own decision-theoretic one.  Specifically, we
shall show that, when the above conditions hold, so do those of
\secref{grec} (see also Theorem 3.2 of \textcite{jmr:lns}).

\condref{eval} is just \condref{ac2} specialized to the case of the
deterministic strategy $\ce$.

To continue, we construct a fictitious distribution $p_i(\cdot)$
$(i=0, \ldots, N)$, for variables $(L_1, A_1,\ldots,L_N,A_N,Y)$, as
follows.

\begin{defn}
  \label{def:fict}
  The distribution $p_i$ of $(L_1, A_1,\ldots,L_N,A_N,Y)$ is defined
  as the distribution under $p$ of $(L_{\co,1}, A_{\co,1}, \ldots,
  L_{\co,i}, A_{\co,i}, L_{\co,i+1}, g(\overline L_{\co,i+1}), \ldots,
  L_{\co,N}, g(\overline L_{\co,N}), Y_\ce)$.
\end{defn}
Thus
\begin{eqnarray}
  \nonumber
  \lefteqn{p_i(\overline L_N = \overline l_N, \overline A_N = \overline a_N, Y=y) :=}\\
  \label{eq:fict}
  &&  \left\{\begin{array}[c]{ll}
      p(\overline L_{\co,N} = \overline l_N, \overline A_{\co,i} = \overline a_i, Y_\ce = y) &
      \mbox{if } a_{i+1} = g(\overline l_{i+1}), \ldots,  a_{N} = g(\overline l_{N})\\
      0 & \mbox{otherwise.}
    \end{array}
  \right.
\end{eqnarray}

Note that this construction of $p_i$ is quite different from that
developed, in a different context, in \secref{extstab}.  In
particular, the marginal joint distribution of $(\overline L_N)$ is,
for every $p_i$, the same as under $p_\co$.

Equation~\eqref{0=e} follows trivially from \defref{fict}.

As in \secref{extstab}, Properties~\eqref{acnew}, and \eqref{l} for $i
\leq N$, hold because the joint distribution of all variables up to
and including $L_i$ is the same under $p_{i-1}$ as under
$p(\,\cdot\,;\,\co)$; while for \eqref{l} with $i = N+1$, when
$L_{N+1} \equiv Y$, we also use \condref{consistency1}.
  
Equation~\eqref{a} holds since the distribution on either side is
concentrated on $g(\overline l_i)$.

Finally we show \eqref{y}.

We only need this for $(\overline l_i, \overline a_i) \in
\Gamma_{i-1}$.  Since then $(\overline l_i, \overline a_i) \in
\Gamma$, we must by \eqref{gamma} have $p(a_j \mid \overline l_j,
\overline a_{j-1}\,;\,\ce)> 0$ ($1 \leq j \leq i$), which in virtue of
the deterministic nature of strategy $\ce$ requires $\overline a_i =
\overline g(\overline l_i)$; and then the additional condition
$(\overline l_i, \overline a_i) \in {\cal O}$ becomes $p(\overline
L_{\co,i} = \overline l_i, \overline A_{\co,i} = \overline g(\overline
l_i)) > 0$.  So in this case \eqref{y} becomes:
\begin{equation}
  \label{eq:lhsy}
  p(Y_\ce = y \mid \overline L_{\co,i} = \overline l_i, \overline A_{\co,i-1} =
  \overline g(\overline l_{i-1})) =     p(Y_\ce = y \mid \overline
  L_{\co,i} = \overline l_i, \overline A_{\co,i} = 
  \overline g(\overline l_{i})). 
\end{equation}
But this is an immediate consequence of \condref{ignore1}.

In summary, we have shown:
\begin{thm}
  Under Conditions~\ref{cond:ac1a}, \ref{cond:consistency1} and
  \ref{cond:ignore1}, and defining $p_i(\cdot)$ by \defref{fict},
  Conditions~\ref{def:positive} and \ref{cond:ac2} and
  equations~\eqref{0=e}--\eqref{y} are all satisfied.
\end{thm}

From \lemref{induction} we now deduce:

\begin{cor}
  Under Conditions~\ref{cond:ignore1}--\ref{cond:eval}, the
  consequence of strategy $\ce$ can be recovered using the
  $G$-recursion \eqref{determstandrecur}.
\end{cor}

\section{Discussion}
\label{sec:discuss}

\subsection{What has been achieved?}
\label{sec:what}
In this work we have described and developed a fully
decision-theoretic approach to the problem of dynamic treatment
assignment.  The central issue identified and addressed is the
transfer of probabilistic information between differing regimes.  When
justified, this can allow future policy analysis to take appropriate
account of previously gathered data.

Out of this approach we have developed an alternative derivation and
interpretation of Robins's $G$-computation algorithm, relating it to
the fundamental `backward induction' recursion algorithm of dynamic
programming.  Moreover we have shown that this is applicable more
generally, including to problems involving randomized treatment
decisions.

We have devoted some attention to the question of how one might
justify the simple stability property \eqref{expci2}, or the more
general conditions of \lemref{induction}.  One can attempt this by
including unobservable variables into one's reasoning, and using
influence diagram to check the desired properties by simple graphical
manipulations.  However, as discussed in \secref{seqirrgraph}, the
graphical approach sometimes imposes more restrictions than necessary,
and an algebraic approach based on manipulations of conditional
independence properties can be more general and powerful.

We have also broadened the application of the graphical approach of
\textcite{pearl/robins:uai95} to allow assessment of the effects of
conditional interventions, that are allowed to depend on the values of
other variables in the problem. This is a particularly natural
requirement when we contemplate sequential interventions, where it is
clearly desirable to be able to respond appropriately to the
information obtained to date, and so naturally to consider dynamic
strategies.  We have noted that the graphical expression of
condition~\eqref{expci2} for simple stability is equivalent to
sequential application of Pearl's back-door criterion, and that this
allows identification by $G$-recursion of the consequences of
conditional interventions, not only for the ultimate response $Y$ but
also for every intermediate covariate $L_i$.  We have further noted
that our graphical check for the more general case of \secref{moregen}
is equivalent to that suggested by \textcite{jmr:lns}.

\subsection{Syntax and semantics}
\label{sec:syntax}

An important pragmatic aspect of our approach is that, in order to
apply it sensibly, we have to be very clear about the real-world
meaning of all the variables (whether `random' or `decision')
appearing in our formulae.  Thus, when considering some interventional
regime, we need to understand exactly what real-world interventions
are involved: we can not assume that setting a variable to a specific
value in different ways, or in different contexts, will have the same
overall effects on the system studied --- see
\textcite{hernan/taubman:08} for a discussion of these issues in the
context of a specific application.  Whenever we consider arguments in
favour of or against accepting a condition such as stability or
extended stability, we must do so in full appreciation of the applied
context and circumstances --- there can be no purely formal way of
addressing such issues.

This emphasis on the semantics of our representations contrasts with
that of other popular approaches, such as causal interpretation of
DAGs or the {\em do\/}-calculus \cite{pearl:book}, which appear to
operate purely syntactically.  However that is an illusion, since
those interpretations and manipulations are always grounded in an
already assumed formal representation of the problem (\eg\ as a DAG,
or a set of structural equations).  So until we have satisfied
ourselves that this representation truly does capture our
understanding of the real-world behaviour of our problem --- in
particular, that it correctly describes the effects of the
interventions we care about --- there can be no reason to have any
faith in the results of any formal manipulations on it.

\subsection{Statistical inference}

We have not directly addressed problems of statistical inference.  One
might want to estimate the consequences of some proposed sequential
strategy, or test a null hypothesis that no strategy is effective in
controlling the outcome.  In principle one can estimate the
ingredients of the $G$-recursion formula, either parametrically or
non-parametrically, from the available data, and then (assuming simple
stability, or the more general conditions of \lemref{induction}) apply
it to supply estimates or tests of the effects of strategies of
interest.  The proposal by \textcite{arjas/saarela:10} can be regarded
as a Bayesian version of $G$-computation.  However, as pointed out by
\textcite{jr+lw:97}, na\"ive use of parametric models for the required
conditional distributions can lead to a `null-paradox', rendering it
impossible to discover that different strategies have the same
consequences.  Also, when continuous variables are included,
$G$-recursion can involve a large number of nested integrals and
become computationally impossible to implement.  Hence we find only a
few instances where $G$-computation has been used for practical data
analysis \cite{robins/greenland/hu:99,taubmanetal:09}.  The problems
in applying $G$-recursion are exacerbated by the need, in many
practical applications, to choose a large set of covariates ${\cal L}$
so as to justify the stability assumption. This makes the modelling
task more difficult and raises issues of robustness to
misspecification. Such considerations have motivated the introduction
of marginal or nested `structural models'
\cite{robins:enc98,robins/hernan/brumback:marg_struct}, as well as
doubly-robust methods \cite{kang/schafer:07}, avoiding the
null--paradox.  Note that while $G$-recursion provides a
likelihood-based approach to the estimation of the consequence of a
given strategy, these latter methods rely on estimating equations. It
should be straightforward to reinterpret these models and analyses
within a fully decision-theoretic framework, by appropriate modelling
of the intervention distributions $p(\,\cdot\,;s)$.

\subsection{Optimal dynamic treatment strategies}
Our work is motivated in part by the desire to compare a variety of
sequential treatment strategies so as to identify the best one.
Recall that our set of regimes is given by $\cs=\{\co\} \cup \cs^*$,
where $\co$ is the observational regime, and $\cs^*$ is the set of
interventional strategies that we want to compare.  If we want to
apply $G$-recursion, justifying it by simple stability as in
\secref{grec0} or by the more general conditions of
\lemref{induction}, we need to ensure that the respective conditions
hold for {\em all\/} strategies $\ce\in \cs^*$ that we want to
compare.  As we saw in \secref{without}, this is not trivial: if
$\cs^*$ contains static as well as dynamic strategies, in some
situations the former may be identified while the latter are not.  In
fact it follows from \textcite{apd/vd:uai08} that if want to find an
optimal strategy among all dynamic regimes, we will usually need the
restrictive requirement of simple stability to hold for all $\ce\in
\cs^*$.

As mentioned in \secref{estimation}, the standard dynamic programming
routine for identifying an optimal strategy can be regarded as a
combination of $G$-recursion and stagewise optimisation.  Under
conditions allowing $G$-recursion, this can in principle be put
directly into effect, after estimating all the required distributional
ingredients from the available data.  In practice (as pointed out by
\textcite{jr:mm} and many others since), this quickly becomes
infeasible, especially if one wants to avoid parametric restrictions.
This is because the number of possible histories for which the optimal
next decision has to be determined at each stage of the backward
induction recursion can grow extremely rapidly with increasing number
$N$ of time points and levels of $(\overline l_i, \overline a_{i-1})$.

Alternative approaches to the optimisation problem to sidestep this
computational complexity have been suggested.  \textcite{susan}
introduces a method based on regret functions (see the discussion and
application in \textcite{henderson}), which is closely related to the
structural nested models of \textcite{jmr:04} (see
\textcite{moodie/etal:07} for a comparison of these two approaches).
\textcite{henderson/al:09} modify Murphy's approach so as to be
amenable to standard statistical model checking procedures.  However,
all these alternative methods for finding optimal dynamic treatments
rely on the same identification conditions underlying $G$-computation,
as well as on various additional (semi-)parametric assumptions.

\subsection{Complete identifiability}

Simple stability, or the alternative conditions of \lemref{induction},
are sufficient conditions allowing the use of $G$-recursion, and
thereby identification of the consequences of a given strategy.  In
recent years the Artificial Intelligence community has devoted some
effort to finding necessary as well as sufficient conditions for the
identifiability of consequences of interventions
\cite{huang:ncai06,shpitser:uai06,shpitser:ncai06}.  These results
rely heavily on the assumptions encoded in causal DAGs or
semi-Markovian causal models.  Even within this more restricted
framework, the general question of identifiability of dynamic
treatment strategies seems still to be an open problem (but see
\textcite{tian:08}).

\subsection{Other problems}
\label{sec:other}
Many problems in causal inference, previously tackled using potential
response or causal DAG formulations, gain in clarity, simplicity and
generality when reformulated as problems of decision analysis.
Specific topics that have been fruitfully treated in this way include:
confounding \cite{apd:infdiags}; partial compliance \cite{apd:hsss};
direct and indirect effects \cite{vd/apd/sg:uai,geneletti:07};
identification of the effect of treatment on the treated
\cite{sgg/apd:ett}; Mendelian randomization \cite{did/sheeh:07};
Granger causality \cite{eichler/did:10}; and causal inference under
outcome-dependent sampling \cite{didetal:10}.  However there still
remains a wide range of other issues in `causal inference' that we
believe would benefit from the application of the decision-theoretic
viewpoint.

\section*{Acknowledgment}
We are indebted to Susan Murphy for stimulating this work and for many
valuable comments. We also want to thank Jamie Robins for helpful
discussions.  {Financial support from MRC Collaborative Project Grant
  G0601625 is gratefully acknowledged.}

\section*{APPENDIX}
\appendix

\section{Two lemmas on DAG-separation}
\label{sec:problemmas}

Here we prove generalised versions of equations~(8) and (9) (Lemma~1)
of \textcite{pearl/robins:uai95}.

Let $\dag$ be a DAG.

\begin{lemma}
  \begin{equation}
    \label{eq:prlem1}
    \inda Y X Z {\dag} \Rightarrow  \inda Y X {Z^*} {\dag}
  \end{equation}
  whenever $Z \subseteq Z^* \subseteq \an(X \cup Y \cup Z)$.
\end{lemma}

\begin{proof}
  Let $\graf:=\ma(X \cup Y \cup Z)$; then also $\graf = \ma(X \cup Y
  \cup Z^*)$.  The left-hand side of \eqref{prlem1} says that any path
  from $Y$ to $X$ in $\graf$ intersects $Z$, whence it must also
  intersect the larger set $Z^*$.
\end{proof}

\begin{lemma}
  \label{lem:pr2}
  \begin{equation}
    \label{eq:prlem2}
    \inda Y X Z {\dag} \Rightarrow  \inda Y X {Z^*}  {\dag}
  \end{equation}
  whenever $Z^* = Z \cap A$ for $A$ an ancestral set in $\dag$
  containing $X \cup Y$.
\end{lemma}

\begin{proof}
  We first note that $(X \cup Y) \cup Z^*$ is a subset of $A$, since
  both its terms are.  Since $A$ is ancestral, it follows that
  \begin{equation}
    \label{eq:suba}
    \an(X \cup Y \cup Z^*) \subseteq A.
  \end{equation}
  
  Define $\graf$ as above, and $\graf':=\ma(X \cup Y \cup Z^*)$.  Then
  both the node-set and edge-set for $\graf'$ are subsets of the
  corresponding sets for $\graf$, and hence the same property holds
  for the path-set.  Suppose the right-hand side of \eqref{prlem2}
  fails.  Then there exists a path $\pi$ in $\graf'$ connecting $Y$
  and $X$ and avoiding $Z^*$; then $\pi$ is a path in $\graf$ with the
  same property.  Since $\pi\subseteq \graf'$, if it intersects $Z$
  anywhere it can only do so at a point of $\an(X \cup Y \cup Z^*)$
  --- and thus, by \eqref{suba}, at a point in $A$, and hence in
  $Z^*$.  Since this has been excluded, the result follows.
\end{proof}

\end{document}